\documentclass[12pt,a4paper]{article}
\usepackage[utf8]{inputenc} % set input encoding (not needed with XeLaTeX)
\usepackage[english]{babel}
\usepackage{amsmath}
\usepackage{amsfonts}
\usepackage{amssymb}
\usepackage{url}
\usepackage{graphicx, color}
\usepackage[margin=3cm,footskip=1cm]{geometry}
\usepackage{amsthm}
\usepackage[final]{fixme}
\usepackage{authblk}

\usepackage[mathlines,pagewise]{lineno}
\usepackage[round]{natbib}
\usepackage{enumerate}
\theoremstyle{plain}% default
\newtheorem{thm}{Theorem}[section]
\newtheorem{lem}[thm]{Lemma}

\newtheorem{cor}[thm]{Corollary}
\theoremstyle{definition}

\newtheorem{exmp}{Example}
\newtheorem{rem}[thm]{Remark}
\theoremstyle{remark}

\def\N{\mathbb{N}}
\def\P{\mathbb{P}}

\def\R{\mathbb{R}}

\def\cB{\mathcal{B}}

\def\cF{\mathcal{F}}

\def\cK{\mathcal{K}}

\def\cN{\mathcal{N}}
\def\cP{\mathcal{P}}

\def\cO{\mathcal{O}}

\DeclareMathOperator{\diag}{diag}
\DeclareMathOperator{\intr}{int}
\DeclareMathOperator{\cl}{cl}
\DeclareMathOperator{\sign}{sign}

\DeclareMathOperator{\conv}{conv}
\DeclareMathOperator{\argmin}{argmin}
\DeclareMathOperator{\lin}{lin}
\DeclareMathOperator{\tr}{tr}

\DeclareMathOperator{\dH}{\delta_H}	%Hausdorff distance	

\newcommand{\nn}{\nonumber}

\newcommand{\eL}{\hat{Q}_m}			%Estimator based on Legendre moments on P^(m).
\newcommand{\ePn}{\tilde{P}_{N,n}}  %LSQ based on noisy measurements of Legendre moments.
\newcommand{\eP}{\hat{P}_{N,n}}     %LSQ, no noise on \cPt.
\newcommand{\cPt}{\cP(\theta_1, \dots, \theta_n)}   %Set of consistent polygons.
\newcommand{\sol}{\P_{N,n}(\epsilon)}				%Set of LSQ estimators, noise.

\author[1]{Astrid Kousholt}
\author[2,3]{Julia Schulte\footnote{The second author married and changed the last name from Hörrmann to Schulte during the submission process.}}

\affil[1]{Department of Mathematics, Aarhus University, astrid.kousholt@gmail.com}
\affil[2]{Department of Computer Science, ETH Zürich, jschulte@ethz.ch}
\affil[3]{Department of Mathematics, Ruhr-Universität Bochum, julia.hoerrmann@rub.de}

\date{}

\title{Reconstruction of convex bodies from moments}

\begin{document}
\maketitle

\begin{abstract}
We investigate how much information about a convex body can be retrieved from a finite number of its geometric moments. We give a sufficient condition for a convex body to be uniquely determined by a finite number of its geometric moments, and we show that among all convex bodies, those which are uniquely determined by a finite number of moments form a dense set. Further, we derive a stability result for convex bodies based on geometric moments. It turns out that the stability result is improved considerably by using another set of moments, namely Legendre moments. We present a reconstruction algorithm that approximates a convex body using a finite number of its Legendre moments. The consistency of the algorithm is established using the stability result for Legendre moments. When only noisy measurements of Legendre moments are available, the consistency of the algorithm is established under certain assumptions on the variance of the noise variables.

\end{abstract}


\section{Introduction}
Important characteristics of a compact set $K\subset \R^n$ are its \textit{geometric moments} (sometimes only referred to as moments) where
\begin{equation}\label{eq:GeomMomSet}
\mu_{\alpha}(K) = \int_K x^{\alpha} dx
\end{equation}
 is the geometric moment of order $\lvert \alpha \rvert$ for a multi-index $\alpha \in \N_0^n$, $x^\alpha:=x_1^{\alpha_1}\cdots x_n^{\alpha_n}$ and $\lvert \alpha \rvert:=\alpha_1+\cdots +\alpha_n$.

The reconstruction of a geometric object from its moments is of interest in a wide range of fields among which are probability and statistics \cite{Diaconis.1987}, signal processing \cite{Sezan.1987} and computational tomography \cite{Milanfar.1995,Milanfar.1996}, see \cite{Golub.1999} for an overview.
E.g. the reconstruction of a geometric object from its X-rays can be reformulated as the problem of reconstruction from geometric moments \cite[Section V]{Milanfar.1995}. 

In the last two decades, the reconstruction of a polytope from its moments has received considerable attention.
Milanfar et al. developed in \cite{Milanfar.1995} an inversion algorithm for $2$-dimensional polygons and presented a refined numerically stable version in \cite{Golub.1999}. Restricting to convex polygons they proved that every $m$-gon is uniquely determined by its complex moments up to order $2m-3$.
Recently, Gravin et al. showed in \cite{Gravin.2011} that an $n$-dimensional convex polytope $P$ with $m$ vertices is uniquely determined by its moments up to order $2m-1$.
 %More precisely it suffices the knowledge of moments of the form
%\[
%\mu_j(z_i):=\int\limits_{P}\langle x, z_i\rangle ^jdx,\quad j=0,\ldots, 2m-1,\; i=1,\ldots,n+1,
%\]
%with $z_1,\ldots,z_n$ being linearly independent vectors, $z_{n+1}=\sum\limits_{j=1}^n \alpha_j z_j$ with $\alpha_j>0, 0\leq j\leq n$ and none of the vectors $z_1,\ldots, z_{n+1}$ being orthogonal to any line connecting two vertices of $P$.
Apart from polytopes, an exact reconstruction from finitely many moments is known to be possible for so-called quadrature domains in the complex plane, see \cite{Gustafsson.2000}. 

Furthermore, the reconstruction of convex bodies from different kinds of indirect measurements has seen several advances \cite{Gardner.2006} including the 
reconstruction from measurements of the support function \cite{Prince.1990}, the brightness function \cite{Gardner.2003} or the lightness function \cite{Campi.2012}. Recently, a similar investigation of another set of moments, namely moments of surface area measures was carried out for planar convex bodies in \cite{Kousholt2016} and for $n$-dimensional convex bodies in \cite{Kousholt2017}.

In continuation of the work in this area, we investigate how much information can be retrieved from finitely many geometric moments of an arbitrary convex body in $\R^n$. 

Using uniqueness results for functionals, see \cite{Krein.1977} and \cite{Putinar.1997}, applied to indicator functions, we show that if a convex body $K$ is of the form $C \cap \{p \geq 0\}$, where  $C$ is a compact subset of $\R^n$ and $p$ is a polynomial of degree $N$, then $K$ is uniquely determined by its geometric moments up to degree $N$ among all convex bodies in $C$. Further, any convex body in $C$ can be approximated arbitrarily well in the Hausdorff metric by a convex body of the form $C \cap \{p \geq 0\}$. This result and the fact that the geometric moments up to order $2$ of a convex body $K$ determine an upper bound on the circumradius of $K$ imply that among all convex bodies, those which are uniquely determined by finitely many geometric moments form a dense subset, see Theorem~\ref{ThmUnique}.  

Restricting to convex bodies in the two-dimensional unit square, we derive an upper bound on the Nikodym distance between two convex bodies given finitely many of their geometric moments, see Theorem~\ref{ThmEstimateFiniteMoments}. The upper bound is derived using a stability result for absolutely continuous functions on the unit interval, see \cite{Talenti.1987}. This result is extended to twice continuously differentiable functions on the two-dimensional unit square and applied to differences of indicator functions via an approximation argument. The upper bound depends on the number of moments used and also on the Euclidean distance between the moments of the two convex bodies. The upper bound decreases when the distance between the moments decreases. However, it increases exponentially in the number of moments. The method used to derive the upper bound of the Nikodym distance suggests that the geometric moments should be replaced by another set of moments, namely the Legendre moments, in order to remove the effect of the exponential factor. The Legendre moments of a convex body are defined like the usual geometric moments, but with the monomials replaced by products of Legendre polynomials, see Section~\ref{SecPrelim}. Using that these products of Legendre polynomials constitute an orthonormal basis of the square integrable functions on the unit square and that the Legendre polynomials satisfy a certain differential equation, we derive an upper bound of the Nikodym distance that becomes arbitrarily small when the distance between the Legendre moments decreases and the number of moments used increases, see Theorem~\ref{ThmEstimateLegendre}.

In Section~\ref{Sec_LSQ}, we assume that the first $(N+1)^2$ Legendre moments of an unknown convex body $K$ are available for some $N \in \N$. A polygon $P$ with at most $m \in \N$ vertices is called a least squares estimator of $K$ if the Legendre moments of $P$ fit the available Legendre moments of $K$ in a least squares sense. We derive an upper bound of the Euclidean distance between the Legendre moments of $K$ and the Legendre moments of an arbitrary least squares estimator $P$ of $K$. In combination with the previously described stability result, this yields an upper bound of the Nikodym distance between $K$ and $P$ (Theorem~\ref{ThmLSQdistbounds}). This upper bound of the Nikodym distance becomes arbitrarily small when $N$ and $m$ increase. For completeness, we further derive an upper bound for the Nikodym distance between $K$ and a least squares estimator based on geometric moments. Due to the structure of the stability results, this upper bound increases exponentially when the number of available geometric moments increases.

In Section~\ref{SecReconstruction}, we derive a reconstruction algorithm for convex bodies. The input of the algorithm is a finite number of Legendre moments of a convex body $K$, and the output of the algorithm is a polygon $P$ with Legendre moments that best fit the available Legendre moments of $K$ in a least squares sense. The output polygon $P$ has prescribed outer normals, which ensures that $P$ can be found as the solution to a polynomial optimization problem. The consistency of the reconstruction algorithm is established in Corollary~\ref{cor_convergence}. In Section~\ref{SecNoise}, the reconstruction algorithm is extended such that it allows for Legendre moments disrupted by noise. To ensure consistency of the algorithm in this case, the variances of the noise terms should decrease appropriately when the number of input moments increases, see Theorem~\ref{ThmNoise}. 
In Section~\ref{Chapter:Implementation} the implementation of the reconstruction from geometric and Legendre moments, respectively, is described and three examples of reconstructions (of a square, a half circle and a body of constant width) are provided. 

The paper is organized as follows. Preliminaries and notations are introduced in Section~\ref{SecPrelim}. The uniqueness results are presented in Section~\ref{SecUnique}, and the stability results are derived in Section~\ref{SecStab}. In Section~\ref{Sec_LSQ}, the least squares estimators based on geometric moments and Legendre moments are treated. Finally, the reconstruction algorithm is described and discussed  in Section~\ref{SecReconstruction} and examples are provided in Section~\ref{Chapter:Implementation}.

\section{Notation and preliminaries}\label{SecPrelim}

We denote by $\N_0$ the natural numbers including $0$ and by $\mathbf{1}_{A}$ the indicator function for a subset $A\subset\R^n$.
In the following we introduce several notions from convex geometry and refer to \cite{Schneider14} as a general reference.
A convex body is a compact, convex subset of $\R^n$ with nonempty interior. The space of convex bodies contained in $\R^n$ is denoted by $\cK^n$ and is equipped with the Hausdorff metric $\dH$. On the set $\{K \in \cK^2 \mid K \subset [0,1]^2\}$, we use the Nikodym distance $\delta_N$ in addition to the Hausdorff metric.  The Nikodym distance of two convex bodies $K,L \subset [0,1]^2$ is the area of the symmetric difference of $K$ and $L$, that is
\begin{equation*}
\delta_N(K,L)= V_2( (K \setminus L) \cup (L \setminus K))=\| \mathbf{1}_K - \mathbf{1}_L \|_2^2,
\end{equation*}
where $\| \cdot \|_2$ is the usual norm on the set $L^2([0,1]^2)$ of square integrable functions on $[0,1]^2$. On the set $\{K \in \cK^2 \mid K \subset [0,1]^2\}$, the Hausdorff metric and the Nikodym distance induce the same topology, see \cite{Shephard1965}. The support function of a convex body $K$ is denoted by $h_K$, and for $K \in \cK^2$ and $\theta \in [0,2\pi)$, we write $h_K(\theta):=h_K((\cos\theta, \sin\theta))$. For a convex body $K$, let $s(K)$ denote the center of mass, $V_n(K)$ the volume, and $R(K)$ the circumradius of $K$. We denote by $B^n$ the closed unit ball and by $S^{n-1}$ the unit sphere in $\R^n$. The volume of the unit ball is $\kappa_n=\frac{\pi^{\frac{n}{2}}}{\Gamma\left(1+\frac{n}{2}\right)}$ and the area of the unit sphere $S^{n-1}$ is $\omega_n=n\kappa_n$. For $u\in S^{n-1}$ and $t\in\R$ we define a hyperplane by $H(u,t)=\{x\in\R^n: x^\top u= t\}$.
A convex body $K\in\cK^n$ is of class $C^k(\cK^n)$ if its boundary $\partial K$ is a $k$-times differentiable regular submanifold and $K\in C^\infty(\cK^n)$ if it is in $C^k(\cK^n)$ for all $k\in\N$. For $K\in C^2(\cK^n)$ we denote by $\kappa_i(K,x)$ for $i=1,\ldots,n-1$ the principal curvatures of $K$ at $x\in\partial K$.

For an open subset $D\subset\R^n$ denote by $C^\infty(D)$ the class of infinitely differentiable real-valued functions on $D$.

 In Section~\ref{SecStab}, we derive stability results for convex bodies in the unit square. In this context, it turns out to be natural and useful to introduce Legendre moments in addition to geometric moments. Let $\langle\cdot,\cdot\rangle$ be the scalar product on $L^2([0,1]^2)$. The shifted and normalized Legendre polynomials $L_i \colon [0,1] \rightarrow \R, i \in \N_0$, are obtained by applying the Gram-Schmidt orthonormalization to $1,x,x^2, \dots $, and the products of Legendre polynomials
\[
(x_1,x_2)\to L_i(x_1)L_j(x_2), \quad i,j\in\N_0
\]
form an orthonormal basis of $L^2([0,1]^2)$. For a convex body $K \subset [0,1]^2$, we define the Legendre moments of $K$ as
\begin{equation*}
\lambda_{ij}(K)= \int_{K} L_i(x_1) L_j(x_2) d(x_1,x_2)
\end{equation*}
for $i,j \in \N_0$.

The uniqueness and stability results we establish in Sections~\ref{SecUnique} and \ref{SecStab} are derived using uniqueness and stability results  \cite{Krein.1977, Putinar.1997, Talenti.1987} for functionals. In the following, we introduce notation in relation to these results.
 For a compact set $C\subset \R^n$ with nonempty interior, we let $L^\infty(C)$ denote the space of essentially bounded measurable functions $\Phi:C\rightarrow \R$. The essential supremum of $\Phi \in L^{\infty}(C)$ over $C$ is denoted by $\|\Phi\|_{\infty,C}$ and we define $\|\Phi\|_{1,C}:=\int_C |\Phi(x)|dx$. Further, we let 
\[\sign(\Phi)(x)=\begin{cases}1,&\Phi(x)\geq 0,\\ -1,&\text{ otherwise} \end{cases}\]
for $x\in\R^n$.
 
The signed distance function $d_C$ of $C$ is defined as in, e.g., \cite[Section 5]{Delfour.1994} or with opposite signs in \cite[Chapter 4.4]{Krantz.2002}. That is 
\[
d_C(x):=\begin{cases}-\inf_{y \in \partial C} \left\| x - y \right\|,&x\in C,\\ \inf_{y \in \partial C} \left\| x - y \right\|,&x \in \R^n \setminus C,
\end{cases}
\]
where $\left\| \cdot \right\|$ is the Euclidean norm on $\R^n$.
Then the $\varepsilon $-parallel set of $C$ is defined as $C_\varepsilon :=\{x: d_C(x)\leq \varepsilon \}$ for $\varepsilon \in \R$.

The geometric moments of a function $\Phi\in L^\infty (C)$ are given as
\begin{equation}\label{eq:GeomMomFunc}
\mu_{\alpha}(\Phi)=\int\limits_{C}\Phi(x) \, x^\alpha dx 
\end{equation}
for  $\alpha \in \N_0^n$, and the Legendre moments of $\Psi \in L^{\infty}([0,1]^2)$ are defined as
\[
\lambda_{ij}(\Psi)=\int\limits_{[0,1]^2}\Psi(x) \, L_i(x_1) L_j(x_2) d(x_1,x_2)
\]
for $i,j \in \N_0$. Notice that $\mu_{\alpha}(K)=\mu_{\alpha}(\mathbf{1}_K)$ for a convex body $K \subset C$, and $\lambda_{ij}(L)=\lambda_{ij}(\mathbf{1}_L)$ for a convex body $L \subset [0,1]^2$.

Furthermore, we need the Steiner formula \cite[Equation (4.1)]{Schneider14} for the volume of the parallel set of a convex body $K$. Let $\epsilon>0$, then
\begin{equation}\label{eq:Steiner}
V_n(K_\epsilon)=\sum\limits_{j=0}^n \epsilon^{n-j}\kappa_{n-j}V_j(K),
\end{equation}
with constants $V_0(K),\ldots,V_n(K)$ which are the intrinsic volumes of $K$. In particular $V_n(K)$ is the volume of $K$ as previously defined, $V_{n-1}(K)$ is half of the $(n-1)$-dimensional surface area and $V_0(K)=1$.

\section{Uniqueness results}\label{SecUnique}
In this section, we present uniqueness results for convex bodies based on a finite number of geometric moments. We show that the convex bodies that are uniquely determined in $\cK^n$ by a finite number of geometric moments form a dense subset of $\cK^n.$ This result is established using uniqueness results from \cite{Krein.1977} and \cite{Putinar.1997} for functionals. The results from \cite{Krein.1977} and \cite{Putinar.1997} are summarized in Section~\ref{SecSummary} and applied in Section~\ref{SecConsequences} to derive uniqueness results for convex bodies.

\subsection{Summary of results from \cite{Krein.1977} and \cite{Putinar.1997}}\label{SecSummary}
Let $N\in\N_0$, $L>0$ and $C \subset \R^n$ be compact. Further, let $m:= (m_\alpha)_{|\alpha|\leq N}$, where $m_\alpha\in\R, \alpha\in\N_0^n$ with $|\alpha|\leq N$ and $\sum_{|\alpha|\leq N} m_\alpha^2>0$.
A function $\Phi\in L^\infty(C)$ with $\|\Phi\|_{\infty,C}\leq L$ is called a solution of the \emph{$L$-moment problem of order $N$} if
\begin{equation}\label{L-momentproblem}
\mu_\alpha(\Phi)=m_\alpha,\quad\alpha\in\N_0^n \text{ with } |\alpha|\leq N.
\end{equation}
In \cite{Krein.1977}, it is shown that the supremum
\begin{equation*}%\label{Definitionlm}
l(m):=\sup\left\{\left\|\sum\limits_{|\alpha|\leq N} a_\alpha x^\alpha\right\|_{1,C}^{-1}:a_\alpha\in\R, \alpha\in\N_0^n, |\alpha|\leq N, \sum\limits_{|\alpha|\leq N}a_\alpha m_\alpha=1
\right\}
\end{equation*}
is attained. Thus, there exists an $\tilde{a}=(\tilde{a}_\alpha)_{|\alpha|\leq N}$ with $\sum\limits_{|\alpha|\leq N} \tilde{a}_\alpha m_\alpha =1$ and
\[l(m)=\left\|\sum\limits_{|\alpha|\leq N} \tilde{a}_\alpha x^\alpha\right\|_{1,C}^{-1}.\]
It follows from \cite{Krein.1977} that the $L$-moment problem \eqref{L-momentproblem} has a solution if and only if $L\geq l(m)$. Furthermore, \eqref{L-momentproblem} has a unique solution if and only if $L=l(m)$. If $L=l(m)$, then the unique solution is $\Phi= L\sign(p_m)$, where $p_m=\sum_{|\alpha|\leq N} \tilde{a}_\alpha x^\alpha$.

For more details and proofs, we refer to \cite[Section IX.1-2]{Krein.1977} and \cite{Putinar.1997}. The one-dimensional case is proved in \cite[Section IX.2, Thm. 2.2]{Krein.1977} by applying more general results from \cite[Section IX.1]{Krein.1977} which are obtained in normed linear spaces with moments defined with respect to arbitrary linear independent functionals instead of monomials. The specialization of \cite[Section IX.1]{Krein.1977} to the situation considered above is contained in \cite[Section 2]{Putinar.1997}. In particular, Putinar \cite{Putinar.1997} formulates the following uniqueness result.
\begin{lem}[{\cite[Cor.2.3]{Putinar.1997}}] \label{CorPutinar}
Let $C\subset\R^n$ be compact. A function $\Phi\in L^\infty(C)$ is uniquely determined in $\{\Psi\in L^\infty(C):\|\Psi\|_{\infty,C}\leq \|\Phi\|_{\infty,C}\}$ by its geometric moments $\mu_\alpha(\Phi)$, $\alpha\in\N_0^n$ with $|\alpha|\leq N$ if and only if
\[
\Phi=\|\Phi\|_{\infty,C} \sign(p),
\]
where $p\neq 0$ is a polynomial of degree at most $N$. 
\end{lem}

\subsection{Consequences for convex bodies}\label{SecConsequences}
Due to the relation between geometric moments of convex bodies and geometric moments of indicator functions, we conclude the following from Lemma~\ref{CorPutinar}.
\begin{cor}\label{CorPutForConvBod}
Let $C\subset \R^n$ be compact. A convex body $K \subset C$ is uniquely determined in $\{L\in\cK^n: L\subset C \}$ by its geometric moments $\mu_\alpha(K)$, $\alpha\in\N_0^n$ with $|\alpha|\leq N$ if
\[
K= C\cap \{p\geq 0\},
\]
where $p\neq 0$ is a polynomial of degree at most $N$.
\end{cor}
%-------------------------------------------------------------------------
\begin{proof}
If $K=C\cap \{p\geq 0\}$, then
\[2\cdot\mathbf{1}_K(x)-1=\sign(p)(x), \quad x\in C.\]
We apply Lemma~\ref{CorPutinar} with $\Phi=\sign(p)$. 
Since $\|\sign(p)\|_{\infty,C}=1$ we obtain that $2\cdot\mathbf{1}_K-1=\sign(p)$ is uniquely determined
in \[\left\{2\cdot\mathbf{1}_L-1:L\in \cK, L\subset C\right\}\subset\left\{\Psi\in L^\infty(C): \|\Psi\|_{\infty,C}\leq 1\right\}\]
by its geometric moments $\mu_\alpha\left(2\cdot\mathbf{1}_K-1\right)$ with $|\alpha|\leq N$. 
The definitions \eqref{eq:GeomMomSet} and \eqref{eq:GeomMomFunc} of the geometric moments for a compact set and a function, respectively, imply
 $\mu_\alpha\left(2\cdot\mathbf{1}_K-1\right)=2\mu_\alpha(K)-\mu_\alpha(C),\alpha\in\N_0^n$ with $|\alpha|\leq N$. 
Since $C$ is fixed, this yields the assertion.
\end{proof}
%%%%%%%%%%%%%%%%%%%%%%%%%%%%%%%%%%%%%%%%%%%%%%%%%%%%%%%%%%%%%%
\begin{exmp}
An ellipsoid $E$ is determined among all convex bodies by its geometric moments up to order $2$ since $E=\{x\in\R^n: p(x)\geq 0\}$, where
$p(x):=R-\|Tx\|^2$, $x\in\R^n$ with some invertible linear transformation $T$ and some $R>0$.
\end{exmp}
%%%%%%%%%%%%%%%%%%%%%%%%%%%%%%%%%%%%%
\begin{rem}
Corollary~\ref{CorPutForConvBod} gives a sufficient condition for a convex body to be uniquely determined among convex bodies in a prescribed set by a finite number of moments. It is not clear if the condition is also necessary.
\end{rem}
%%%%%%%%%%%%%%%%%%%%%%%%%%%%%%%%%%%%%
\begin{rem}
Let $m:=(m_\alpha)_{|\alpha|\leq N}$ be a finite number of geometric moments of some unknown convex body $K \subset C$. Let $\tilde{m}_{\alpha}:=2m_\alpha-\mu_\alpha(C)$ and define $l(\tilde{m})$ and $p_{\tilde{m}}$ as in the previous section. Then it holds that $l(\tilde{m})\leq 1$, and if $l(\tilde{m})=1$, then $K=C\cap \{p_{\tilde{m}}\geq 0\}$.
\end{rem}

In Theorems~\ref{UniqinC} and \ref{ThmUnique}, we show that the convex bodies which are uniquely determined among all convex bodies by finitely many geometric moments form a dense subset of $\cK^n$ with respect to the Hausdorff metric $\dH$. The ideas of the proofs are summarized in the following. For a convex body $K\subset C$, a function $f \colon \R^n \to \R$ with $K=C\cap \{f\geq 0\}$ is constructed. The function $f$ is approximated by a polynomial $p_m$ of degree $m$ in such a way that $K_m:=C\cap \{p_m\geq 0\}$ is convex and $\dH(K,K_m)$ is small.
Then, it follows from Corollary \ref{CorPutForConvBod} that $K_m$ is uniquely determined by its geometric moments up to order $m$ among all convex bodies contained in $C$.
The circumradius of $K_m$ admits an upper bound which can be expressed in terms of the geometric moments $\mu_\alpha(K_m), |\alpha|\leq 2$ of $K_m$. Therefore, $K_m$ is uniquely determined by its geometric moments up to order $m$ among all convex bodies if $C$ is large enough.

 Note firstly that we can assume that $K$ is of class $C^\infty_+(\cK^n)$, see \cite[Thm. 3.4.1]{Schneider14} and the subsequent discussion. Hence, the boundary of $K$ is a regular submanifold of $\R^n$ of class $C^k(\cK^n)$ for all $k \in \N_0$. Further, the principal curvatures of $K$ are strictly positive. By $\kappa_{i}(K,x), 1\leq i\leq n-1$ we denote the principal curvatures of $K$ at $x\in\partial K$.
Since $K\in C^\infty_+(\cK^n)$ there exist $m_K,M_K>0$ such that \[\kappa_i(K,x)\in(m_K,M_K),\quad x\in\partial K,1\leq i\leq n-1.\] For $\varepsilon < M_K^{-1}$ it follows from the inverse function theorem applied as in \cite[Lemma 14.16]{Gilbarg.2001} that the signed distance function $d_K$ of $K$ is an infinitely differentiable function in $\R^n \setminus K_{-\varepsilon}$.
%Then, the spherical image map $\nu:\partial K\rightarrow S^{n-1}$ is in $C^\infty(\cK^n)$ and the curvatures are the eigenvalues of its Jacobian matrix $J\nu$. Since $K$ is compact there exist $M_1, M_2>0$ with
%$ \sum\limits_{i=1}^n \kappa_i(K,x)= \tr J\nu (x)\leq M_1$ and
%$\prod\limits_{i=1}^n \kappa_i(K,x)=\det J\nu(x) \geq M_2$. In particular
%\[\frac{M}{M_1^{n-1}}\leq \kappa_i(K,x)\leq M_1,\quad x\in K, 1\leq i\leq n-1.\]
%
As in \cite[Sec. 14.6]{Gilbarg.2001} we define for $y\in\partial K$ the principal coordinate system at $y$ as the coordinate system with coordinate axes $x_1(y),\ldots,x_n(y)$, where $x_1(y),\ldots, x_{n-1}(y)$ are the principal directions  and $x_n(y)$ is the inner unit normal vector of $K$ at $y$.
Then the following lemma is obtained by adapting \cite[Lemma 14.17]{Gilbarg.2001}.
\begin{lem}\label{LemHessianDistance}
Let $K\in C^\infty_+(\cK^n)$ and $M_K>0$ be such that $\kappa_i(K,x)\leq M_K$ for all $x\in\partial K$ and $1\leq i\leq n-1$. Further, let $\varepsilon < M_K^{-1}$, $x_0\in \R^n \setminus K_{-\varepsilon}$ and  $y_0:=\mathop{\argmin}_{y \in \partial K} \left\| x_0- y \right\|.$
Then, with respect to the principal coordinate system at $y_0$, we have
\[
\nabla d_K(x_0)=(0,\ldots ,0,-1)^\top
\]
and
\[\left(\frac{\partial^2}{\partial_i\partial_j}d_K(x_0)\right)_{i,j=1}^n =\diag\left(\frac{\kappa_1(K,y_0)}{1+\kappa_1(K,y_0)d_K(x_0)},\ldots,\frac{\kappa_{n-1}(K,y_0)}{1+\kappa_{n-1}(K,y_0)d_K(x_0)} , 0\right).\]
\end{lem}

By the described approximation argument and Lemma~\ref{LemHessianDistance}, we obtain the following result.

\begin{thm}\label{UniqinC}
Let $K, C$ be convex bodies in $\R^n$ with $K\subset  \intr C$. For $\varepsilon >0$ there exists an $m\in\N$ and a convex body $K_m\subset C$ which is uniquely determined by its geometric moments up to order $m$ among all convex bodies contained in $C$ and fulfils
\[\dH(K,K_m)\leq \varepsilon .\]
\end{thm}
%-------------------------------------------------------------------------------
\begin{proof}
We may assume $K\in C^\infty_+(\cK^n)$, $2\varepsilon  < M_K^{-1}$ and $K_\varepsilon \subset C$.
We have $K=\{f\geq 0\}$   for the function $f:\R^n \rightarrow \R$ defined by
\begin{equation*}%\label{Definitionf}
f(x):=\begin{cases} 1,& x\in K_{-2\varepsilon },\\
-\frac{(d_K(x)+2\varepsilon )^4}{16\varepsilon ^4}+1,& x\in \R^n\setminus K_{-2\varepsilon }.\end{cases}
\end{equation*}
Observe that $f$ is of class $C^3(\R^n)$ and
\begin{equation}\label{Rangef}
f(x)\in [-65/16,15/16]\quad\Leftrightarrow \quad x\in (\partial K)_\varepsilon .
\end{equation}
The Hessian matrix $(\frac{\partial^2}{\partial_i\partial_j}f(x))_{1\leq i,j\leq n}$ is negative definite for $x\in(\partial K)_\varepsilon $.
Namely, let  $x_0\in (\partial K)_\varepsilon $ and  $y_0:=\mathop{\argmin_{y \in \partial K}} \left\| x_0 - y\right \|.$
Then it follows from Lemma \ref{LemHessianDistance} that, with respect to the principal coordinate system at $y_0$,
\begin{align*}
\frac{\partial^2}{\partial_i\partial_j} f(x_0)
%=\frac{\partial^2}{\partial_i\partial_j}[-\frac{(d_K+2\varepsilon )^4}{16\varepsilon ^4}+1](x_0)
%=\frac{\partial}{\partial_i}\left[-4\frac{(d_K+2\varepsilon )^3}{16\varepsilon ^4}\frac{\partial}{\partial_j}d_K\right](x_0)\\
%%
%&=-12 \frac{(d_K(x_0)+2\varepsilon )^2}{16\varepsilon ^4}\frac{\partial}{\partial_i}d_K(x_0) \frac{\partial}{\partial_j}d_K(x_0)
%-4\frac{(d_K(x_0)+2\varepsilon )^3}{16\varepsilon ^4}\frac{\partial^2}{\partial_i\partial_j}d_K(x_0)\\
%%
&=\begin{cases} -\frac{(d_K(x_0)+2\varepsilon )^3}{4\varepsilon ^4}\frac{\kappa_i(K,y_0)}{1+\kappa_i(K,y_0)d_K(x_0)},& i=j<n,\\
-\frac{3(d_K(x_0)+2\varepsilon )^2}{4\varepsilon ^4},&i=j=n, \\
0 & i \neq j.
\end{cases}
\end{align*}
Therefore, the eigenvalues of the Hessian matrix $\left(\frac{\partial^2}{\partial_i\partial_j}f(x)\right)$ for $x\in(\partial K)_{\varepsilon }$ are all negative and their absolute values are uniformly bounded from below by
\begin{equation}\label{EigenvalueBound}
\min\left\{\frac{ m_K}{4\varepsilon (1+M_K\varepsilon )}, \frac{3}{4\varepsilon ^2} \right\}.
\end{equation}

From \cite[Thm. 2]{Bagby.2002}, we obtain that
for every $m\geq 2$ there exists a polynomial $p_m$ of degree $m$ such that
\begin{equation}\label{PolynomialApprox}
\left\|\frac{\partial^{\lvert \alpha \rvert}}{\partial_1^{\alpha_1} \cdots \partial_n^{\alpha_n}}(f-p_m)\right\|_{\infty,C}\leq c(n,C,f) \frac{1}{m^{3-|\alpha|}},\quad \alpha\in\N_0^n \text{ with } |\alpha|\leq 2,
\end{equation}
where $c(n,C,f)>0$ depends on $n$, $C$ and $\max_{|\alpha|\leq 3}\|\frac{\partial^{\lvert \alpha \rvert}}{\partial_1^{\alpha_1} \cdots \partial_n^{\alpha_n}} f\|_{\infty,C}$.

As the function that maps a symmetric matrix to its eigenvalues is Lipschitz continuous (\cite[Thm. VI.2.1]{Bhatia.1997}), the bounds~\eqref{EigenvalueBound} and \eqref{PolynomialApprox} imply that the Hessian matrix $\left(\frac{\partial^2}{\partial_i\partial_j}p_m\right)_{1\leq i,j\leq n}$ of $p_m$ is negative definite on
$(\partial K)_\varepsilon $ if we choose $m\geq 2$ sufficiently large. Thus, by the well-known convexity criterion \cite[Thm. 1.5.13]{Schneider14}, the polynomial $p_m$ is concave on every convex subset of $(\partial K)_\varepsilon $.
Let
\[  K_m:=C\cap \{p_m\geq 0\},\]
then it follows
\begin{equation}\label{ConcaveOnBoundNeigh}
\left( p_m(x),p_m(y)\geq 0 \text{ for } x,y\in (\partial K)_\varepsilon  \text{ with } [x,y]\subset (\partial K)_\varepsilon  \right) \Rightarrow [x,y]\subset K_m.
 \end{equation}

Due to $\eqref{PolynomialApprox}$, we can furthermore assume that
\[
\|f-p_m\|_{\infty,C}< 15/16.
\]
Then it follows from \eqref{Rangef} that $p_m\geq0$ on $K_{-\varepsilon }$ and $p_m<0$ on $C\setminus K_\varepsilon $. In other words, we have $K_{-\varepsilon }\subset K_m\subset K_{\varepsilon }$. This implies that $\dH(K,K_m)\leq \varepsilon $ since $\varepsilon < M_K^{-1}$. 

Furthermore, we can show that $K_m$ is convex by distinguishing the following four cases. Let $x,y \in K_m $. 
\begin{enumerate}
\item If $x,y\in K_{-\varepsilon }$, then $[x,y]\subset K_{-\varepsilon }$ since $K_{-\varepsilon }$ is convex and thus $[x,y]\subset K_m$. 
\item If $x\in K_{-\varepsilon }$ and $y\in (\partial K)_\varepsilon $, there is a $z\in [x,y]$ such that
$[x,z]\subset K_{-\varepsilon }$ and $[z,y]\subset (\partial K)_\varepsilon $. Hence, it follows from \eqref{ConcaveOnBoundNeigh} that $[x,y]\subset K_m$.  
\item If $x,y\in (\partial K)_\varepsilon $ and $[x,y]\subset (\partial K)_\varepsilon $, then $[x,y]\subset K_m$ because of \eqref{ConcaveOnBoundNeigh}. 
\item If $x,y \in (\partial K)_\varepsilon $ and $[x,y]\cap K_{-\varepsilon }\neq \emptyset$, there are $z_1, z_2\in [x,y]$ such that $[x,z_1]\subset (\partial K)_\varepsilon , [z_1,z_2]\subset K_{-\varepsilon }$ and $[z_2,y]\subset (\partial K)_\varepsilon $. Then, it follows again from the convexity of $K_{-\varepsilon }$ and by \eqref{ConcaveOnBoundNeigh} that $[x,y]\subset K_m$.
\end{enumerate}
\end{proof}
%%%%%%%%%%%%%%%%%%%%%%%%%%%%%%%%%%%%%%%%%%%%%%%%%%%%%%%%%%%%%%%%%%%%%%%%%%%%%%%
We define
\[
\tilde{K}:=V_n(K)^{-1/n} (K-s(K)).
\]
As $V_n(\tilde{K})=1$ and $s(\tilde{K})=0$, a special case of {\cite[Lem. 4.1]{Paouris.2003}} yields that
\begin{equation*}%\label{SupFctEst}
\left(\int\limits_{\tilde{{K}}}|\langle x, u\rangle|^2 dx\right)^{1/2} \geq \left(\frac{\Gamma(3)\Gamma(n)}{2e \Gamma(n+3)}\right)^{1/2}
\max\{h_{\tilde{K}}(u),h_{\tilde{K}}(-u)\}
\end{equation*}
for $u\in S^{n-1}$. Then the Cauchy-Schwarz inequality implies that 
\[\tilde{K}\subset  I_2(\tilde{K})\left(\frac{\Gamma(3)\Gamma(n)}{2e \Gamma(n+3)}\right)^{-1/2} B^n,\]
where
\begin{equation*}
I_2(L):=\left( \int_L \lVert x \rVert^2 dx \right)^{\frac{1}{2}}
\end{equation*} 
for a convex body $L$. Since $R(K)=(V_n(K))^{1/n}R(\tilde{K})$, we obtain an upper bound
\begin{equation}\label{CircumEst}
R(K)\leq \left(\frac{2e \Gamma(n+3)}{\Gamma(3)\Gamma(n)}\right)^{1/2}
V_n(K)^{1/n}I_2(\tilde{K})
\end{equation}
of the circumradius of $K$.
%%%%%%%%%%%%%%%%%%%%%%%%%%%%%%%%%%%%%%%%%%%%%%%%%%%%%%%%%%%%%%%%%%%%%%%%%%%
\begin{lem}\label{ReprOfI2}
Let $K\in\cK^n$, then $I_2(\tilde{K})$ can be expressed in terms of the geometric moments of $K$ up to order $2$ by
\[
I_2(\tilde{K})=\mu_0(K)^{-\frac{1+n}{n}}\left(\sum\limits_{j=1}^n \mu_0(K)\mu_{2e_j}(K)-\mu_{e_j}(K)^2 \right)^{1/2},
\]
where $\{e_1,\ldots,e_n\}$ is the standard basis in $\R^n$.
\end{lem}
%------------------------------------------------------------------------------
\begin{proof}

For $\alpha>0$ and $\beta\in\R^n$ the transformation formula and the definition of the geometric moments imply 
\begin{align*}
I_2(\alpha(K-\beta))&=\alpha^{\frac{n}{2}+1}\left(\sum\limits_{j=1}^n \mu_{2 e_j}(K)-2\beta_j\mu_{e_j}(K) +\beta_j^2\mu_0(K)\right)^{\frac{1}{2}}.
\end{align*}
Furthermore, it holds
\[
V_n(K)=\mu_0(K)
\]
and the $j$th coordinate of the center of mass of $K$ is 
\begin{align*}
s(K)_j&= \frac{1}{V_n(K)}\int\limits_{K} x_j dx= \frac{\mu_{e_j}(K)}{\mu_0(K)}.
\end{align*}
Thus, we obtain the assertion by choosing $\alpha=V_n(K)^{-\frac{1}{n}}$ and $\beta=s(K)$.
\end{proof}
%%%%%%%%%%%%%%%%%%%%%%%%%%%%%%%%%%%%%%%%%%%%%%%%%%%%%%%%%%%%%%%%%%%%%%%%%%%%%%%%

The previous considerations allow us to formulate a strengthened version of Theorem~\ref{UniqinC} for the whole class of convex bodies and not only those contained in a prescribed compact set.  

\begin{thm}\label{ThmUnique}
Let $K$ be a convex body in $\R^n$. For $\varepsilon >0$ there exists an $m\in\N$ and a convex body $K_m$ which is uniquely determined by its geometric moments up to order $m$ among all convex bodies and fulfils
\[\dH(K,K_m)\leq \varepsilon .\]
\end{thm}
%-------------------------------------------------------------------------------
\begin{proof}
Without loss of generality we may assume that $K\in C^\infty_+(\cK^n)$ and $V_n(K_{-\varepsilon })>0$.
Let\[
c(K,\varepsilon ):= \left(\frac{e \Gamma(n+3)}{\Gamma(3)\Gamma(n)}
\frac{2^{n+3} \, \omega_n}{n+2} V_n(K_{-\varepsilon})^{-1}R(K_{\varepsilon})^{n+2}\right)^{1/2},
\]
and choose \begin{equation}\label{RCond}
R>c(K,\varepsilon )
\end{equation}
such that $K \subset RB^n$. By Theorem~\ref{UniqinC} there exists an $m\in\N$ and a convex body $K_m\subset (3R + \varepsilon) B^n$ which is uniquely determined by its geometric moments up to order $m$ among all convex bodies contained in $(3R + \varepsilon) B^n$ and fulfils
\begin{equation}\label{Assump1}
\dH(K,K_m)\leq \varepsilon .
\end{equation}
Due to the proof of Theorem~\ref{UniqinC}, we can assume that $m\geq 2$ and $K_{-\varepsilon }\subset K_m\subset K_\varepsilon $.
Then, condition \eqref{RCond} ensures that $K_m$ is uniquely determined among all convex bodies. Namely, let $L$ be a convex body with 
\begin{equation}\label{Assump2}
\mu_{\alpha}(L)=\mu_{\alpha}(K_m),\quad \alpha\in\N_0^n, \text{ with }|\alpha|\leq m.
\end{equation}
Then it follows by Lemma \ref{ReprOfI2} and a simple calculation that
\begin{align*}
I_2(\tilde{L})&= I_2(\tilde{K}_m)\leq I_2(2 V_n(K_m)^{-1/n}R(K_m) B^n)
\\
&=\bigg(\frac{2^{n+2} \, \omega_n}{n+2} V_n(K_m)^{-(n+2)/n}R(K_m)^{n+2}\bigg)^{1/2},
\end{align*}
where we have used that $\tilde{K}_m \subset 2R(\tilde{K}_m)B^n$ as $s(\tilde{K}_m)=0$.
Thus, we obtain by \eqref{CircumEst} that
\begin{equation*}
 R(L)\leq   \left(\frac{e \Gamma(n+3)}{\Gamma(3)\Gamma(n)}
\frac{2^{n+3} \, \omega_n}{n+2} V_n(K_m)^{-1}R(K_m)^{n+2}\right)^{1/2} \leq c(K,\varepsilon ).
\end{equation*}
Assumption~\eqref{Assump2} implies that $s(L)=s(K_m)$, so
\begin{equation*}
\sup_{x \in L} \| x \| \leq \sup_{x \in L} \| x - s(L) \| + \| s(K_m) \| \leq 3R + \varepsilon
\end{equation*}
as $K_m \subset (R+\varepsilon)B^n$ by \eqref{Assump1}. Thus, $L \subset (3R+\varepsilon)B^n$, so $K_m=L$, and we obtain the assertion.
\end{proof}

\begin{rem}
Due to the one-to-one correspondence between the geometric moments up to order $m$ and the Legendre moments up to order $m$ of a convex body, the uniqueness results stated in this section
 hold if the geometric moments are replaced by Legendre moments in the two-dimensional case.
\end{rem}

%%%%%%%%%%%%%%%%%%%%%%%%%%%%%%%%%%%%%%%%%%%%%%%%%%%%%%%%%%%%%%%%%%%%%%%%%%%%%%%%

%%%%%%%%%%%%%%%%%%%%%%%%%%%%%%%%%%%%%%%%%%%%%%%%%%%%%%%%%%%%%%%%

\section{Stability results}\label{SecStab}

In this section, we derive stability results for two-dimensional convex bodies contained in the unit square. We derive an upper bound for the Nikodym distance of convex bodies where the first $(N+1)^2$ moments are close in the Euclidean distance. The stability results are based on more general results for twice continuously differentiable functions on the unit square.

\subsection{Stability results for functions on the unit square}
The study in this section uses ideas from \cite{Talenti.1987} (see also \cite{Ang1999}), which considers the problem of recovering a real-valued function $f$ defined on the interval $(0,1)$ from its first $N+1$ moments $\mu_0(f),\ldots,\mu_N(f)$. In \cite{Talenti.1987}, it is shown that if $f, g:(0,1)\to \R$ are absolutely continuous functions satisfying
\[
\sum\limits_{k=0}^N |\mu_k(f)-\mu_k(g)|^2\leq\varepsilon^2
\]
and
\[
\|f^\prime(x)-g^\prime(x)\|_2^2 \leq E^2
\]
for some $\varepsilon, E>0$, then
\[
\|f-g\|_2^2\leq \min\left\{\frac{\varepsilon^2}{E^2} e^{3.5(n+1)}+ \frac{1}{4}(n+1)^{-2}:n=0,\ldots,N\right\}.
\]
Using the same ideas as \cite{Talenti.1987}, we deduce the following corresponding theorem in two dimensions.
\begin{thm}\label{StabResFunc}
If $g,h\in C^2([0,1]^2)$ are twice continuously differentiable functions satisfying
\[
\sum\limits_{i,j=0}^N |\mu_{ij}(g)-\mu_{ij}(h)|^2\leq \varepsilon^2
\]
and
\[
\frac{1}{4} \left\|\frac{\partial}{\partial x_1}(g-h)\right\|_2^2
+\frac{1}{4} \left\|\frac{\partial}{\partial x_2}(g-h)\right\|_2^2\leq E^2
\]
for some $\varepsilon, E>0$, then
\[
\|g-h\|_2^2\leq \min \{a_0(n+1)^2 e^{7(n+1)}\varepsilon^2 + (n+1)^{-2}E^2: n=0,\ldots,N \}
\]
where $a_0 > 0$.
\end{thm}
\begin{proof}
Let $f_N$ be the orthogonal projection of $f:=g-h$ on the linear hull $\lin \{x_1^{i}x_2^{j}: i,j=0,\ldots,N\}$ with respect to the usual scalar product on $L^2([0,1]^2)$. Furthermore, let
\[t_N:=f-f_N\]
be the projection of $f$ on the orthogonal complement of $\lin \{x_1^{i}x_2^{j}: i,j=0,\ldots,N\}$. Then
\[f_N(x_1,x_2)= \sum\limits_{i,j=0}^N \lambda_{ij}(f) L_i(x_1)L_j(x_2)\]
and
\[
t_N(x_1,x_2)= \sum\limits_{\substack{i,j=0\\ i\vee j>N}}^\infty\lambda_{ij}(f) L_i(x_1)L_j(x_2),
\]
where $\lambda_{ij}(f), i,j \in \N_0$ are the Legendre moments of $f$. For $i\in\N_0$, the coefficients of the polynomial $L_i$ are denoted by $C_{ij}$, $ j=0,\ldots,i,$ that is
\[
L_i(x)=\sum\limits_{j=0}^{i} C_{ij}x^j, \quad x\in [0,1].
\]
Then it follows for $i,j=0,\ldots,N$ that
\begin{align}\label{FormulaForLegendreMoments}
 \lambda_{ij}(f) %&=  \int\limits_{[0,1]^2} u(x_1,x_2)L_i(x_1)L_j(x_2)d(x_1,x_2)
 & = \sum\limits_{k=0}^{i}\sum\limits_{l=0}^{j} C_{ik}C_{jl}\int\limits_{[0,1]^2}f(x_1,x_2)x_1^k x_2^l d(x_1,x_2) \nn \\
 & = \sum\limits_{k=0}^{i}\sum\limits_{l=0}^{j} C_{ik} C_{jl} \mu_{kl}(f) \nn \\
 & = (C M C^\top)_{ij},
\end{align}
with
\[
C := \left(
      \begin{array}{cccc}
        C_{00} &  & &  \\
        C_{10} & C_{11} &  &  \\
        \vdots &  & \ddots &  \\
        C_{N0} & C_{N1} & \ldots & C_{NN} \\
      \end{array}
    \right) \quad \text{ and } \quad M:= (\mu_{ij}(f))_{i,j=0,\ldots,N}.
\]

The Frobenius norm of a square matrix A is defined as $|A|_F:=\sqrt{\tr(A^\top A)}$, and since this norm is submultiplicative, see \cite[(3.3.4)]{Johnson.1990}, we obtain that
\begin{align}
\|f_N\|_2&=\sqrt{\sum\limits_{i,j=0}^N \lambda_{ij}(f)^2} = \sqrt{\tr(L^\top L)}\nn\\
& = |L|_F = |C M C^\top|_F\leq |C|_F |M|_F |C^\top|_F\nn\\
& = |C|_F^2 |M|_F\label{eqhNAbschaetzung}
\end{align}
where $L:=(\lambda_{ij}(f))_{i,j=0,\ldots,N}$. The matrix $C^{\top}C$ has $N+1$ non-negative eigenvalues, $0\leq l_0\leq l_1\leq \ldots \leq l_N$, and $C^\top C= H_N^{-1}$, where $H_N$ is the Hilbert matrix
\[H_N:= \left(\frac{1}{i+j+1}\right)_{i,j=0,\ldots,N},\]
see \cite[(22)]{Talenti.1987}. Since $\|H_N e_1\| > 1$, the Hilbert matrix $H_N$ has an eigenvalue larger than $1$, so the smallest eigenvalue $l_0$ of $H_N^{-1}$ is smaller than $1$. This implies that
\begin{equation}
|C|_F^2= \tr(C^\top C)= \sum\limits_{i=0}^N l_i \leq (N+1)l_N \leq (N+1) \frac{l_N}{l_0} \approx a_0(N+1)e^{3.5(N+1)}\label{estimateFrobeniusNormC}
\end{equation}
with a constant $a_0 >0$, where we have used the approximation \cite[(8)]{Talenti.1987}, see also \cite[p. 111]{Taussky.1954}. From equation \eqref{eqhNAbschaetzung} and \eqref{estimateFrobeniusNormC}, we obtain that
\begin{equation}
\|f_N\|_2\leq a_0(N+1)e^{3.5(N+1)} \sqrt{\sum\limits_{i,j=0}^N \mu_{ij}(f)^2}.\label{eqEstimateHN}
\end{equation}

The shifted Legendre polynomials satisfy the differential equation
\begin{equation*}
- \frac{\partial}{\partial x_1}[x_1(1-x_1)L_i^\prime (x_1)]=i(i+1)L_i(x_1),\quad x_1\in[0,1], i\in\N_0,%\label{partDifOneDimLegendre}
\end{equation*}
see \cite[(25)]{Talenti.1987}.
From this differential equation, we obtain by multiplication with $f(x_1,x_2)$, integration over $[0,1]$ with respect to $x_1$ and twofold integration by parts that
\begin{multline}
-\int\limits_{[0,1]}L_i(x_1)\frac{\partial}{\partial x_1}\left[ x_1(1-x_1)\frac{\partial}{\partial x_1}f(x_1,x_2)\right] dx_1
=\\ i(i+1)\int\limits_{[0,1]}L_i(x_1)f(x_1,x_2)dx_1.\label{EqFromDifLeg1}
\end{multline}
By multiplication with $L_j(x_2)$ and integration with respect to $x_2$, it follows from $\eqref{EqFromDifLeg1}$ that
\begin{multline*}
-\int\limits_{[0,1]^2}L_i(x_1)L_j(x_2)\frac{\partial}{\partial x_1}\left[ x_1(1-x_1)\frac{\partial}{\partial x_1}f(x_1,x_2)\right] dx_1 dx_2
=\\
 i(i+1)\int\limits_{[0,1]^2}L_i(x_1)L_j(x_2)f(x_1,x_2)dx_1 dx_2.
\end{multline*}
This implies that the Legendre moments of the function
\[
\tilde{f}: (x_1,x_2)\mapsto \frac{\partial}{\partial x_1}\left[ x_1(1-x_1)\frac{\partial}{\partial x_1}f(x_1,x_2)\right]
\]
are equal to $-i(i+1)\lambda_{ij}(f),i,j\in\N_0$. Thus, we obtain from the theory of Hilbert spaces that
\[
-\sum\limits_{i,j=0}^\infty i(i+1)\lambda_{ij}(f)^2 =\sum\limits_{i,j=0}^\infty\lambda_{ij}(\tilde{f})\lambda_{ij}(f) =\sum\limits_{i,j=0}^\infty \langle \tilde{f},L_i L_j\rangle \langle L_i L_j,f\rangle=\langle \tilde{f},f\rangle.
\]
This and integration by parts yield that
\begin{align}
\sum\limits_{i,j=0}^\infty i(i+1)\lambda_{ij}(f)^2  &= \int\limits_{[0,1]^2} \frac{\partial}{\partial x_1}\left[ x_1(1-x_1)\frac{\partial}{\partial x_1}f(x_1,x_2)\right](-f(x_1,x_2)) d(x_1,x_2)\nn\\
& =  \int\limits_{[0,1]^2} x_1(1-x_1)\left(\frac{\partial}{\partial x_1}f(x_1,x_2)\right)^2 d(x_1,x_2)\nn\\
& \leq \frac{1}{4} \left\|\frac{\partial}{\partial x_1}f(x_1,x_2)\right\|_2^2.\label{eqEst2}
\end{align}
In the same way, we conclude that
\begin{equation}
\sum\limits_{i,j=0}^\infty j(j+1) \lambda_{ij}(f)^2\leq \frac{1}{4} \left\|\frac{\partial}{\partial x_2}f(x_1,x_2)\right\|_2^2.\label{eqEst3}
\end{equation}
The inequalities~\eqref{eqEst2} and \eqref{eqEst3} imply that
\begin{align}
\|t_N\|_2^2&= \sum\limits_{\substack{i,j=0\\ i\vee j>N}}^\infty \lambda_{ij}(f)^2
 \leq\sum\limits_{i=N+1}^\infty\sum\limits_{j=0}^\infty  \lambda_{ij}(f)^2 +\sum\limits_{i=0}^\infty \sum\limits_{j=N+1}^\infty \lambda_{ij}(f)^2 \nn\\
& \leq \sum\limits_{i,j=0}^\infty \frac{i(i+1)}{(N+1)^2}\lambda_{ij}(f)^2+
\sum\limits_{i,j=0}^\infty \frac{j(j+1)}{(N+1)^2}\lambda_{ij}(f)^2\nn\\
&\leq \frac{1}{(N+1)^2} \bigg(\frac{1}{4} \left\|\frac{\partial}{\partial x_1}f(x_1,x_2)\right\|_2^2 +\frac{1}{4} \left\|\frac{\partial}{\partial x_2}f(x_1,x_2)\right\|_2^2\bigg),\label{eqEstTN}
\end{align}
and as a consequence we obtain that
\[
\|g-h\|_2^2 = \|f_N\|_2^2+\|t_N\|_2^2\leq a_0(N+1)^2 e^{7(N+1)}\varepsilon^2 + \frac{1}{(N+1)^2}E^2.
\]
\end{proof}

\subsection{Application to Convex Bodies}
In this section, we approximate the indicator function $\mathbf{1}_K$ of a convex body $K$ by a smooth function and apply the result from the previous section. In this way, we obtain an estimate for the Nikodym distance of two convex bodies in terms of the Euclidean distance of their first $(N+1)^2$ geometric moments.

\begin{thm}\label{ThmEstimateFiniteMoments}
If $K,L\subset [0,1]^2$ are convex bodies satisfying
\[
\sum\limits_{i,j=0}^N |\mu_{ij}(K)-\mu_{ij}(L)|^2\leq \varepsilon^2,
\]
for some $\varepsilon \geq 0$, then
\[
\delta_N(K,L)\leq
 \min \left\{a_0(n+1)^2 e^{7(n+1)}\varepsilon^2 + \frac{a_1}{(n+1)} : n=0,\ldots,N \right\},
\]
with constants $a_0,a_1 > 0$.
\end{thm}
\begin{proof}
Let $f:=\mathbf{1}_K-\mathbf{1}_L$. As in the proof of Theorem~\ref{StabResFunc}, we let $f_N$ denote the orthogonal projection of $f$ on $\lin\{x_1^{i}x_2^{j}:i,j=0,\ldots,N\}$ and let $t_N$ denote the projection on the orthogonal complement of $\lin\{x_1^{i}x_2^{j}:i,j=0,\ldots,N\}$. In the proof of Theorem~\ref{StabResFunc}, the smoothness of $u$ is not used when the estimate \eqref{eqEstimateHN} is derived. Therefore, we obtain in the same way that
\begin{equation}
\|f_N\|_2\leq a_0(N+1)e^{3.5(N+1)} \sqrt{\sum\limits_{i,j=0}^N \mu_{ij}(f)^2}.\label{eqEstimateHN2}
\end{equation}

Using a mollification, see \cite[p. 110]{Miklavcic.2001}, we obtain for every $\rho>0$ a differentiable function $f^{(\rho)}:[0,1]^2 \rightarrow \mathbb{R}$ approximating $f$ in the $L^1$-norm. More precisely, we choose
\[
f^{(\rho)}(x) := (J_\rho \ast f)(x)=\int\limits_{[0,1]^2}J_\rho(x-y)f(y)dy
,\quad x\in [0,1]^2,
\]
where
\[
J_\rho= \begin{cases}\begin{array}{rl} c_0 \rho^{-2} e^{-\frac{\rho^2}{\rho^2- \|x\|^2}},& \text{ for }\|x\|<\rho\\ 0,& \text{ for } \|x\|\geq\rho\end{array}\end{cases}
\]
with a constant $c_0>0$ chosen such that $|J_\rho|_{L^1(\mathbb{R}^2)} =1.$ Notice that $c_0$ is independent of $\rho$ and that $J_\rho\in C^\infty(\R^2)$.
From the definition of the mollification, we obtain that
\[
\|f-f^{(\rho)}\|_\infty \leq \|f\|_{\infty }+ |J_\rho|_{L^1(\R^2)} \|f\|_\infty
\leq 2
\]
and
\[
(f-f^{(\rho)})(x)=0,\quad x\in A:=\left[K_{-\rho}\cup([0,1]^2\setminus K_{\rho})\right]\cap \left[L_{-\rho}\cup([0,1]^2\setminus L_{\rho})\right]
\]
since for $x\in A$ the function $f=\mathbf{1}_K-\mathbf{1}_L$ is constant on $x+\rho B^2$ which implies $f^{(\rho)}(x)=f(x)$.
Thus
\begin{align*}
\|f-f^{(\rho)}\|_2^2
&\leq \|f-f^{(\rho)}\|_\infty^2\; V_2(\left[(K_{\rho}\setminus K)\cup(K\setminus K_{-\rho})\right]\cup
\left[(L_{\rho}\setminus L)\cup(L\setminus L_{-\rho})\right] )
 \\
& \leq 4\left[V_2(K_{\rho} \setminus K)+ V_2(K\setminus K_{-\rho})+
V_2(L_{\rho}\setminus L)+ V_2(L\setminus L_{-\rho})
\right]
\end{align*}
for $\rho \in (0,1)$. Then, the fact that
\[ V_2(K\setminus K_{-\rho}) \leq V_2(K_{\rho}\setminus K),\] the Steiner formula \eqref{eq:Steiner}, and the monotonicity of the intrinsic volumes imply that
\begin{align*}
\|f-f^{(\rho)}\|_2^2
&\leq 8 \left[V_2(K_{\rho}\setminus K)+V_2(L_{\rho}\setminus L)\right]\\
& \leq 8 \left[ 2\rho^2 \pi+2\rho (V_1(K)+V_1(L)) \right]\\
&\leq  (16\pi+64)\rho \leq 11^2\rho ,
\end{align*}
where $V_1$ is the intrinsic volume of order $1$, so $V_1(M)$ is half the boundary length of a convex body $M$.
For $\rho \in (0,1)$, let $t_N^{(\rho)}$ be the orthogonal projection of $f^{(\rho)}$ on the orthogonal complement of $\lin\{x_1^{i} x_2^{j}: i,j=0,\ldots,N\}.$
Then it follows from Pythagoras' theorem and \eqref{eqEstTN} that
\begin{align*}
\|t_N\|_2 &\leq
\|t_N-t_N^{(\rho)}\|_2+\|t_N^{(\rho)}\|_2
\\
&\leq
\|f-f^{(\rho)}\|_2 +\frac{1}{N+1}E_{\rho}
\\
&\leq 11\sqrt{\rho}+\frac{1}{N+1}E_{\rho}
,
\end{align*}
where $E_{\rho}>0$ is some constant satisfying
\begin{equation}\label{eqEdelta}
\frac{1}{4} \left\|\frac{\partial}{\partial x_2}f^{(\rho)}\right\|_2^2
+\frac{1}{4} \left\|\frac{\partial}{\partial x_1}f^{(\rho)}\right\|_2^2\leq E_{\rho}^2.
\end{equation}
In order to obtain an expression for a constant $E_{\rho}$ that satisfies \eqref{eqEdelta}, we first observe that
\[
\frac{\partial}{\partial x_2}f^{(\rho)}(y)=\frac{\partial}{\partial x_1}f^{(\rho)}(y)=0, \quad y\in \left[K_{-\rho}\cup ((0,1)^2\setminus K_\rho)\right]\cap\left[L_{-\rho}\cup ((0,1)^2\setminus L_\rho)\right].
\]
Furthermore,
\begin{align*}
\frac{\partial}{\partial x_1}f^{(\rho)}(x)& = \left(\left[\frac{\partial}{\partial x_1}J_\rho\right]\ast f\right)(x) = \int\limits_{[0,1]^2}\left[\frac{\partial}{\partial x_1}J_\rho\right](x-y)f(y)dy\\
&\leq   \int\limits_{\R^2}\left|\frac{\partial}{\partial x_1}J_\rho(y)\right|dy
 = \rho^{-3} \int\limits_{\R^2}\left|\left[\frac{\partial}{\partial x_1}J_1\right]\left(\frac{y}{\rho}\right)\right|dy\\
& = \rho^{-1}\int\limits_{\R^2}\left|\left[\frac{\partial}{\partial x_1}J_1\right](y)\right|dy
\leq c_1 \rho^{-1}
\end{align*}
for $x \in \R^2$ and a constant $c_1 > 0$ independent of $\rho$. It follows that
\begin{align*}
\left\|\frac{\partial}{\partial x_1}f^{(\rho)}\right\|_2^2&\leq  c_1^2 \rho^{-2} \Big[V_2(K_{\rho}\setminus K)+ V_2(K\setminus K_{-\rho})+V_2(L_{\rho}\setminus L)+ V_2(L\setminus L_{-\rho})\Big]\\
&\leq \frac{c_1^2 11^2}{4\rho}.
\end{align*}
In the same way, we obtain
\[
\left\|\frac{\partial}{\partial x_2}f^{(\rho)}\right\|_2^2\leq  \frac{c_2^2   11^2}{4\rho}
\]
for a suitable $c_2>0$ independent of $\rho$.
Therefore, we can choose
\[
E_{\rho}^2:=c_3 \rho^{-1}
\]
for $\rho\in(0,1)$ and some constant $c_3>0$ independent of $\rho$. Letting $\rho=(N+1)^{-1}$, we obtain that
\begin{align*}
\|t_N\|_2^2 &\leq \left(11\sqrt{\rho}+\frac{1}{N+1} \sqrt{c_3} \rho ^{-1/2}\right)^2\\
&= 11^2 \rho + 22\sqrt{c_3} \frac{1}{N+1} + \frac{c_3}{(N+1)^2 \rho}\\
&= (11^2 + 22\sqrt{c_3}+ c_3) \frac{1}{N+1},
\end{align*}
which leads to the assertion.
\end{proof}

The matrix $C$ defined in the proof of Theorem~\ref{StabResFunc} is ill-conditioned and introduces an exponential factor in the upper bound for the Nikodym distance derived in Theorem~\ref{ThmEstimateFiniteMoments}. If the geometric moments are replaced by Legendre moments, the use of the matrix $C$ is avoided and the upper bound can be improved. 

\begin{thm}\label{ThmEstimateLegendre}
If $K, L \subset [0,1]^2$ are convex bodies satisfying
\begin{equation}\label{assumpLegendre}
\sum_{i,j=0}^N \lvert \lambda_{ij}(K)-\lambda_{ij}(L) \rvert^2 \leq \varepsilon^2 
\end{equation}
for some $\varepsilon \geq 0$, then
\begin{equation}\label{UpperBoundN}
\delta_N(K,L) \leq \varepsilon^2 + \frac{a_1}{N+1}
\end{equation}
with a constant $a_1 > 0$.
\end{thm}
The proof of Theorem~\ref{ThmEstimateLegendre} follows the lines of the proof of Theorem~\ref{ThmEstimateFiniteMoments}. Due to inequality~\eqref{assumpLegendre}, the upper bound on the $L^2$-norm of $f_N$ in \eqref{eqEstimateHN2} can be replaced by $\varepsilon$. This yields the upper bound \eqref{UpperBoundN} of the Nikodym distance.

\begin{rem}
If the first $(N+1)^2$ geometric moments of two convex bodies $K, L \subset [0,1]^2$ are identical, then the first $(N+1)^2$ Legendre moments of $K$ and $L$ are identical. In this case, Theorem~\ref{ThmEstimateFiniteMoments} (or Theorem~\ref{ThmEstimateLegendre}) implies that $\delta_N(K,L) \leq \frac{a_1}{N+1}$.
\end{rem}

\begin{rem}\label{NikodymExt}
The Nikodym distance $\delta_N$ is extended in the natural way to the set of compact, convex subsets of the unit square. It then defines a pseudometric, which we also denote by $\delta_N$. As the proofs of Theorems~\ref{ThmEstimateFiniteMoments} and \ref{ThmEstimateLegendre} do not use that the interior of the convex bodies are nonempty, the stability results hold for compact, convex subsets of the unit square and the pseudometric $\delta_N$. In the following sections, we repeatedly consider the distance $\delta_N(K, P_k)$ for a convex body $K \subset [0,1]^2$ and a sequence of polygons $(P_k)_{k \in \N}$ contained in $[0,1]^2$, see Theorems \ref{ThmLSQdistbounds}, \ref{Stability2} and \ref{ThmNoise}. If $\delta_N(K, P_k) \to 0$ for $k \to \infty$, then $\intr P_k \neq \emptyset$ for $k$ sufficiently large. This implies that $\delta_N$ in the expression $\delta_N(K, P_k)$ is a proper metric for sufficiently large $k$ .    
\end{rem}

\section{Least Squares Estimators based on Moments}\label{Sec_LSQ}
Let $K\subset [0,1]^2$ be a convex body and assume that its geometric moments $\mu_{ij}(K)$ for $i,j\in \N_0$ are given. For $m\geq 3$, let $\cP^{(m)}$ denote the set of convex polygons contained in $[0,1]^2$ with at most $m$ vertices. Any polygon $\hat{P}_m \in \cP^{(m)}$ satisfying
\[
\hat{P}_m = \argmin \left\{ \sum\limits_{i,j=0}^N (\mu_{ij}(K)-\mu_{ij}(P))^2:\quad P\in\cP^{(m)} \right\} 
\]
is called a least squares estimator of $K$ with respect to the first $(N+1)^2$ geometric moments on the space $\cP^{(m)}$, where $N \in \N_0$. Likewise, we define a least squares estimator based on the Legendre moments. Assume that the Legendre moments $\lambda_{ij}(K), i,j \in \N_0$ of $K$ are given. Then, any polygon $\eL \in \cP^{(m)}$ satisfying 
\[
\eL= \argmin \left\{ \sum\limits_{i,j=0}^N (\lambda_{ij}(K)-\lambda_{ij}(P))^2:\quad P\in\cP^{(m)} \right\} 
\]
is called a least squares estimator of $K$ with respect to the first $(N+1)^2$ Legendre moments on the space $\cP^{(m)}$. Since the polygons in $\cP^{(m)}$ are uniformly bounded, Blaschke's selection theorem ensures the existence of least squares estimators $\hat{P}_m$ and $\eL$.

\begin{thm}\label{ThmLSQdistbounds}
Let $\hat{P}_m$ and $\eL$ be least squares estimators of $K$ on the space $\cP^{(m)}$ with respect to the first $(N+1)^2$ geometric moments and the first $(N+1)^2$ Legendre moments. Then
\[
\delta_N (\hat{P}_m,K)\leq
\left(a_0(n+1)^2 e^{7(n+1)}\left(1+\frac{1}{2}\ln(2n+1)\right)^2 \frac{8\pi^3 + 16\pi}{m^2}+ \frac{a_1}{(n+1)}\right)
\]
for $n=0, \dots, N$ and
\[
\delta_N (\eL,K)\leq
\frac{8\pi^3 + 16\pi}{m^2}+ \frac{a_1}{(N+1)}.
\]
\end{thm}
%-------------------------------------------------------
\begin{proof}
Let $P \in \cP^{(m)}$ and define $f:=\mathbf{1}_P-\mathbf{1}_K$. Using the notation $f_N, C, L, M $ and $H_N$ from the proof of Theorem~\ref{StabResFunc}, we obtain that
\begin{align*}
& \sqrt{\sum\limits_{i,j=0}^N (\mu_{ij}(P)-\mu_{ij}(K))^2}\\
& = |M|_F = |C^{-1}L (C^{-1})^\top|_F \leq |C^{-1}|^2_F |L|_F\\
& \leq \left(1+\frac{1}{2}\ln(2N+1)\right) \|f_N\|_2\leq
\left(1+\frac{1}{2}\ln(2N+1)\right) \|f\|_2\\
& = \left(1+\frac{1}{2}\ln(2N+1)\right) \sqrt{\delta_N(K,P)}
,
\end{align*}
where we have used that
\begin{align*}
|C^{-1}|^2_F&=\tr(C^{-1}(C^{-1})^\top)= \tr(H_N)
\\
&=\sum\limits_{i=0}^N \frac{1}{2i+1}\leq 1+\int\limits_0^N \frac{1}{2x+1}dx = 1+ \frac{1}{2}\ln(2N+1)
\end{align*}
by the definition of the Hilbert matrix $H_N$. From \cite[p. 730]{Bronstein.2008}, the monotonicity of the intrinsic volumes, and the fact that $\sin(x)\leq x$ for $x\geq 0$, we obtain that
\[
\min\limits_{P\in\cP^{(m)}}\dH(K,P)\leq \frac{V_1(K) \sin(\frac{\pi}{m})}{m(1+\cos(\frac{\pi}{m}))}\leq \frac{2 \pi}{m^2}.
\]
Further, the definition of the Hausdorff metric and the Steiner formula yield that
\begin{align*}
\delta_N(K, P) &\leq V_2((K + \dH(K,P)B^2 )\setminus K) + V_2((P + \dH(K,P)B^2) \setminus P)
\\
&\leq 8\dH(K,P) + 2 \pi \dH(K,P)^2 
\end{align*}
for $P \in \cP^{(m)}$, so
\begin{equation}\label{UniformBound}
\min\limits_{P\in\cP^{(m)}}\delta_N(K,P)\leq \frac{8\pi^3 + 16\pi}{m^2}.
\end{equation}
Therefore,
\begin{align*}
\min\limits_{P\in\cP^
{(m)}} \sum\limits_{i,j=0}^N (\mu_{ij}(P)-\mu_{ij}(K))^2 &\leq \left(1+\frac{1}{2}\ln(2N+1)\right)^2 \min\limits_{P\in\cP^{(m)}}\delta_N(K,P)\\
& \leq \left(1+\frac{1}{2}\ln(2N+1)\right)^2 \frac{8\pi^3 + 16\pi}{m^2}.
\end{align*}
Then, Theorem \ref{ThmEstimateFiniteMoments} and Remark~\ref{NikodymExt} yields that
\[
\delta_N(\hat{P}_m,K) \leq
a_0(n+1)^2 e^{7(n+1)}\left(1+\frac{1}{2}\ln(2n+1)\right)^2 \frac{8\pi^3 + 16\pi}{m^2}+ \frac{a_1}{(n+1)}
\]
for $n=0, \dots, N$. For $P \in \cP^{(m)}$, Parseval's identity yields that 
\begin{equation*}
\sum_{i,j=0}^N (\lambda_{ij}(K) - \lambda_{ij}(P))^2 \leq \|\textbf{1}_K - \textbf{1}_P \|_2^2 =\delta_N(K,P),
\end{equation*}
so we obtain from \eqref{UniformBound}, Theorem~\ref{ThmEstimateLegendre} and Remark~\ref{NikodymExt} that
\begin{equation*}
\delta_N(\eL,K) \leq \frac{8\pi^3 + 16\pi}{m^2} + \frac{a_1}{N+1}
\end{equation*}
for $P \in \cP^{(m)}$.
\end{proof}

In Theorem~\ref{ThmLSQdistbounds}, the upper bound on the distance between the convex body $K$ and the least squares estimator $\hat{P}_m$ based on geometric moments decreases polynomially in the number of vertices $m$, but increases exponentially in the number of moments $N$. However, for the least squares estimator $\eL$ based on Legendre moments, the upper bound decreases polynomially in both $N$ and $m$. Therefore, we concentrate on reconstruction  from Legendre moments in Section~\ref{SecReconstruction}.

\section{Reconstruction based on Legendre moments}\label{SecReconstruction}
In this section, we develop a reconstruction algorithm for a convex body $K\subset [0,1]^2$ based on Legendre moments. To simplify an optimization problem, we approximate $K$ by a polygon with prescribed outer normals. Thus, the input of the algorithm is the first $(N+1)^2$ Legendre moments of $K$ for some $N \in \N_0$, and the output is a polygon $P \subset [0,1]^2$ with prescribed outer normals satisfying that the Euclidean distance between the first $(N+1)^2$ Legendre moments of $P$ and $K$ is minimal.

\subsection{Reconstruction algorithm} \label{Sec_LSQLegendre}
Let $0 \leq \theta_1 < \dots < \theta_n < 2\pi$, and let $c_i:=\cos(\theta_i)$, $s_i:=\sin(\theta_i)$ and  $u_i:=[c_i,s_i]^\top$ for $1 \leq i \leq n$. We assume that
\begin{equation}\label{condition_u}
\left\{\sum\limits_{i=1}^n \lambda_i u_i: \lambda_i\geq 0, 1\leq i\leq n\right\}=\R^2.
\end{equation}
For $h_1,\ldots,h_n\in\R$, let

\[
P(h_1,\ldots,h_n):=\bigcap\limits_{i=1}^n \{x\in \R^2: \langle x,u_i\rangle\leq h_i\}.
\]
A vector $(h_1,\ldots,h_n)\in \R^n$ is called consistent with respect to $(\theta_1, \dots, \theta_n)$ if the polygon $P(h_1,\ldots,h_n)$ has support function value $h_i$ in the direction $u_i$ for $1\leq i\leq n$. In \cite[p. 1696]{Lele.1992}, it is shown that $(h_1,\ldots,h_n)$ is consistent if and only if
\begin{align*}
 h_{i-1}(s_{i+1}c_i-c_{i+1}s_i) - h_i(s_{i+1}c_{i-1}-c_{i+1}s_{i-1})+h_{i+1}(s_{i}c_{i-1}-c_{i}s_{i-1})\geq 0
\end{align*}
for $1\leq i\leq n$, where we define $h_0:=h_n$ and $h_{n+1}:=h_1$. We let $\cPt$ denote the set of polygons $P(h_1, \dots, h_n) \subset [0,1]^2$ where $(h_1, \dots, h_n) \in \R^n$ is consistent with respect to $(\theta_1, \dots, \theta_n)$.

Now let $K\subset [0,1]^2$ be a convex body. Any polygon $\eP \in \cP(\theta_1, \dots, \theta_n)$ satisfying
\[
\eP = \argmin \left\{ \sum\limits_{k,l=0}^N (\lambda_{kl}(K)-\lambda_{kl}(P))^2:\quad P\in\cPt \right\}
\]
is called a least squares estimator of $K$ with respect to the first $(N+1)^2$ moments on the space $\cPt$. As $\cPt$ is closed in the Hausdorff metric, Blaschke's selection theorem ensures the existence of a least squares estimator.

In the following, we let the directions $0 \leq \theta_1 < \dots < \theta_n < 2\pi$ be fixed. We use the notation $s_i, c_i$ and $u_i$ as introduced above and assume that condition~\eqref{condition_u} is satisfied. When $(h_1,\ldots,h_n)\in\R^n$ is consistent with respect to $(\theta_1, \dots, \theta_n)$, we write
\[v_i:=H(u_i,h_i)\cap H(u_{i+1},h_{i+1}),\quad 1\leq i\leq n\]
for the vertices of $P(h_1,\ldots,h_n)$, see Figure~\ref{PolygonWithFixedNormals}.
\begin{figure}[h!]
  \centering
  \includegraphics[width=5cm]{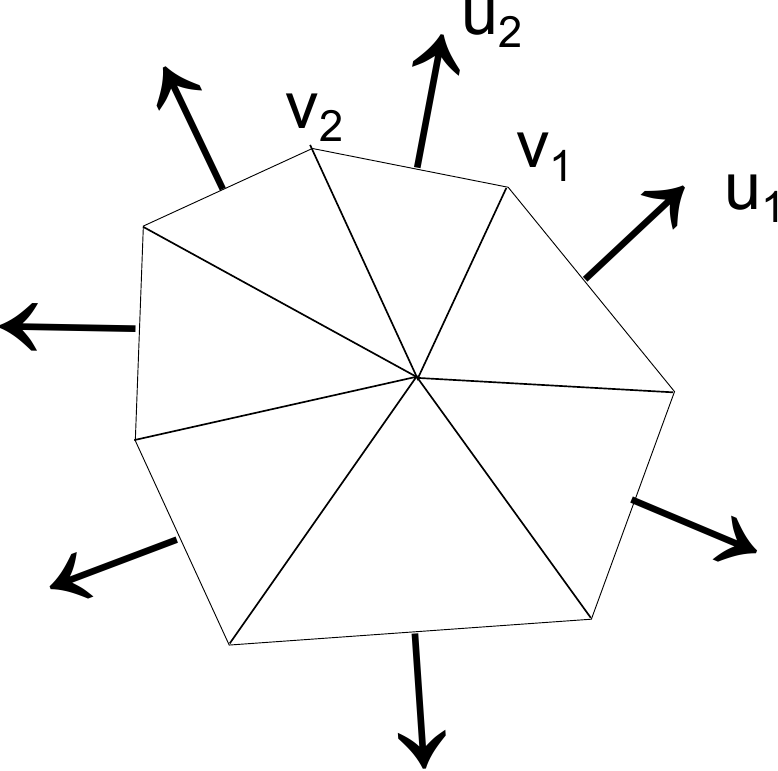}
  \caption{Polygon with normals $u_1,\ldots,u_n$.}\label{PolygonWithFixedNormals}
\end{figure}

In Lemma~\ref{MomentsFromh}, the geometric moments and the Legendre moments of polygons of the form $P(h_1, \dots, h_n)$ are expressed in terms of $(h_1, \dots, h_n)$.
%%%%%%%%%%%%%%%%%%%%%%%%%%%%%%%%%%%%%%%%%%%%%%%%%%%%%%%%%%%%%%%

\begin{lem}\label{MomentsFromh}
Let $(h_1, \dots, h_n) \in \R^n$ be consistent with respect to  $(\theta_1, \dots, \theta_n)$. Then the geometric moments and the Legendre moments of $P(h_1, \dots, h_n)$ are polynomials in $(h_1, \dots, h_n)$. More precisely,
\begin{equation}\label{momentexpr}
\mu_{kl}(P(h_1,\ldots,h_n))= \sum\limits_{i=1}^n \sum\limits_{q_1=0}^{k+l+1}\sum\limits_{q_2=0}^{k+l+2-q_1} M_{kl}(i,q_1,q_2)h_i^{q_1}h_{i+1}^{q_2}h_{i+2}^{k+l+2-q_1-q_2}
\end{equation}
and
\[
 \lambda_{kl}(P(h_1,\ldots,h_n)) =\sum\limits_{i=1}^n \sum_{s=0}^{k+l} \sum\limits_{q_1=0}^{s+1}\sum\limits_{q_2=0}^{s+2-q_1} L_{kl}(i,s,q_1,q_2)h_i^{q_1}h_{i+1}^{q_2}h_{i+2}^{s+2-q_1-q_2},
\]
for $k, l \in \N_0$ and known real constants $M_{kl}(i, q_1, q_2)$ and $L_{kl}(i,s,q_1,q_2)$.
\end{lem}
%%%%%%%%%%%%%%%%%%%%%%%%%%%%%%%%%%%%%%%%%%%%%%%%%%%%%%%%%%%%%%%
\begin{proof}
Observe that
\[
P(h_1,\ldots,h_n)= \cl(A\setminus B),\]
where
\[A:=\bigcup\limits_{\substack{1\leq i\leq n\\ h_{i+1}\geq 0}} \conv\{0,v_{i},v_{i+1}\} \text{ and }B:=\bigcup\limits_{\substack{1\leq i\leq n\\ h_{i+1}<0}}\conv\{0,v_{i},v_{i+1}\} \]
and $v_1, \dots, v_n$ are the vertices of $P(h_1, \dots, h_n)$, see Figure~\ref{RepresentationOfPolygon}. In particular, we have $B\subset A$, so the moments of $P(h_1,\ldots,h_n)$ are equal to the sum
 \begin{equation} \label{momentsum}
 \mu_{kl}(P(h_1,\ldots,h_n))=
 \sum\limits_{i=1}^n \sign(h_{i+1}) \mu_{kl}(\conv\{0,v_i,v_{i+1}\}).
 \end{equation}

  \begin{figure}
  \centering
  \includegraphics[width=5cm]{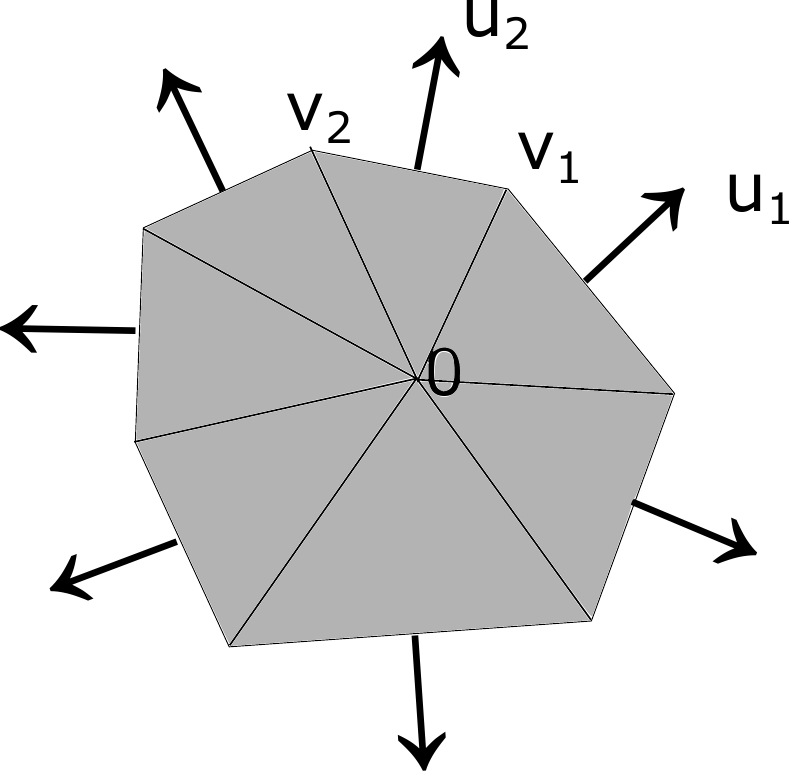}
  \includegraphics[width=5cm]{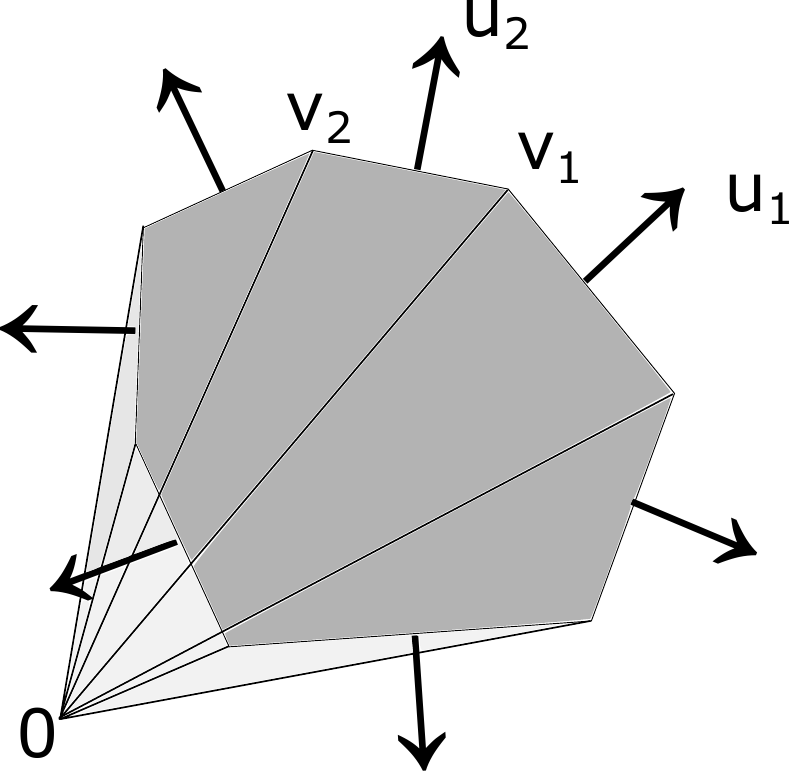}

  \caption{Representation of polygons $P(h_1,\ldots,h_n)$ as difference of the sets $A$ (bright and dark grey) and $B$ (bright grey).}
  \label{RepresentationOfPolygon}
\end{figure}

For $i=1, \dots, n$, let $\overline{u}_i:=(-s_i,c_i)^\top$. Then there exist unique $t_i,\overline{t}_i \in \R$ with 
\begin{equation}\label{vertex}
v_i= h_i u_i +t_i \overline{u}_i = h_{i+1}u_{i+1}-\overline{t}_i\overline{u}_{i+1}.
\end{equation}
This implies that
\begin{equation*}
 \begin{pmatrix}\overline{u}_i\, \overline{u}_{i+1}\end{pmatrix} \begin{pmatrix}t_i\\ \overline{t}_i\end{pmatrix} = h_{i+1}u_{i+1}-h_i u_i,
\end{equation*}
and thus
\begin{align*}
\begin{pmatrix}t_i\\ \overline{t}_i\end{pmatrix}&= \frac{1}{-s_i c_{i+1}+c_i s_{i+1}}
 \begin{pmatrix}c_{i+1}& s_{i+1}\\
 -c_i& -s_i\end{pmatrix} \begin{pmatrix}h_{i+1}c_{i+1}-h_i c_i\\ h_{i+1}s_{i+1}-h_i s_i\end{pmatrix}\\
& =\frac{1}{-s_i c_{i+1}+c_i s_{i+1}} \begin{pmatrix}
h_{i+1}-h_i( c_i c_{i+1}+ s_i s_{i+1})\\h_i-h_{i+1}(c_i c_{i+1}+s_i s_{i+1}) \end{pmatrix}.
 \end{align*}
Substituting this expression of $(t_i\, \overline{t}_i)^{\top}$ into \eqref{vertex}, the vertex $v_i$ can be expressed by $(h_i,h_{i+1})$ and $(u_i, u_{i+1})$. We obtain that
 \begin{equation} \label{vertexexpr}
 v_i = \frac{1}{c_i s_{i+1}-s_i c_{i+1}}\begin{pmatrix}
 h_i s_{i+1}-h_{i+1} s_i\\
h_{i+1}c_i - h_i c_{i+1} \end{pmatrix}.
\end{equation}

Now define $T_i(x_1,x_2):= \begin{pmatrix}v_i, v_{i+1}\end{pmatrix} \begin{pmatrix}x_1\\ x_2\end{pmatrix}= \begin{pmatrix}v_{i,1}x_1+v_{i+1,1}x_2\\ v_{i,2}x_1+v_{i+1,2}x_2\end{pmatrix}.$ Integration by substitution then yields that
 \begin{align*}
 &\mu_{kl}(\conv\{0,v_i,v_{i+1}\})\\
 &= \int\limits_{\conv\{0,v_i,v_{i+1}\}} x_1^{k}x_2^l \, d(x_1,x_2)\\
 & = \int\limits_{\conv\{0,e_1,e_2\}} (v_{i,1}x_1+v_{i+1,1}x_2)^{k} (v_{i,2}x_1+v_{i+1,2}x_2)^l  \lvert v_{i,1}v_{i+1,2}-v_{i,2}v_{i+1,1} \rvert \, d(x_1,x_2).
\end{align*}
Using \eqref{vertex}, the Jacobian determinant of $T_i$ can be expressed as $h_{i+1}(\overline{t}_i + t_{i+1})$, and since $\overline{t}_i + t_{i+1}$ is the length of the facet of $P(h_1, \dots, h_n)$ bounded by $v_i$ and $v_{i+1}$, it follows that
\begin{equation*}
\sign(v_{i,1}v_{i+1,2}-v_{i,2}v_{i+1,1}) = \sign(h_{i+1}).
\end{equation*}
This implies that
\begin{align}
& \sign(h_{i+1}) \mu_{kl}(\conv\{0,v_i,v_{i+1}\}) \notag\\
 & =\int\limits_0^1\int\limits_0^{1-x_2} (v_{i,1}x_1+v_{i+1,1}x_2)^{k} (v_{i,2}x_1+v_{i+1,2}x_2)^l (v_{i,1}v_{i+1,2}-v_{i,2}v_{i+1,1})dx_1 dx_2\label{GeometricMomentTriangle} \\ 
  & =  \sum\limits_{q_1=0}^{k+l+1}\sum\limits_{q_2=0}^{k+l+2-q_1}
 M_{kl}(i,q_1,q_2) h_i^{q_1} h_{i+1}^{q_2} h_{i+2}^{k+l+2-q_1-q_2},\notag
 \end{align}
 with constants $M_{kl}(i,q_1, q_2)$ which are explicitely derived in Lemma \ref{Lem:GeometricMomentCoefficients} using \eqref{vertexexpr}. In combination with \eqref{momentsum}, this yields \eqref{momentexpr}. Furthermore, we obtain from formula \eqref{FormulaForLegendreMoments} for the Legendre moments that
 \begin{align*}
 &\lambda_{kl}(P(h_1,\ldots,h_n))= \sum\limits_{p=0}^{k}\sum\limits_{q=0}^{l} C_{kp} C_{lq} \mu_{pq}(P(h_1,\ldots,h_n))\\
& =   \sum\limits_{s=0}^{k+l}\sum\limits_{q=s-k \vee 0}^{s \wedge l} C_{k,s-q} C_{lq}\sum\limits_{i=1}^n  \sum\limits_{q_1=0}^{s+1}\sum\limits_{q_2=0}^{s+2-q_1} M_{s-q,q}(i,q_1,q_2)h_i^{q_1}h_{i+1}^{q_2}h_{i+2}^{s+2-q_1-q_2}\\
& =\sum\limits_{i=1}^n \sum_{s=0}^{k+l} \sum\limits_{q_1=0}^{s+1}\sum\limits_{q_2=0}^{s+2-q_1} L_{kl}(i,s,q_1,q_2)h_i^{q_1}h_{i+1}^{q_2}h_{i+2}^{s+2-q_1-q_2},
 \end{align*}
 where
 \[
 L_{kl}(i,s,q_1,q_2):=\sum\limits_{q=s-k \vee 0}^{s \wedge l} C_{k,s-q} C_{lq}M_{s-q,q} (i,q_1,q_2).
 \]
 % Expression of the M's. 
% \begin{align}
% & M_{i,k,l}(q_1,q_2):=
% %\sum\limits_{p+q=0\vee (q_1-1)}^{(q_1+q_2-1)\wedge (k+l)}=
% \sum\limits_{p=0\vee(l-q_1+1)}^{k\wedge(q_1+q_2-1) }\sum\limits_{q=0\vee (q_1-1-p)}^{l\wedge(q_1+q_2-p-1)} \binom{k}{p}\binom{l}{q}\frac{1}{(p+q+1)(k+l+2)}\nn\\
% %
% &\times\left(c_i s_{i+1}-s_i c_{i+1}\right)^{-p-q-1}\left( c_{i+1}s_{i+2}-s_{i+1}c_{i+2}\right)^{-k-l+p+q-1}\nn\\
% %
% &\times \Big[\sum\limits_{a_1=0\vee(q_1-q)}^{q_1\wedge (p+1)}\sum\limits_{a_2=0\vee(l-q_1-q_2+p+2)}^{(k-p)\wedge (q_2+q_1-p-q-1)}\binom{p+1}{a_1}\binom{k-p}{a_2}\binom{q}{q_1-a_1}
% \binom{l-q+1}{q_2+q_1-p-q-a_2-1}\nn\\
%  %
%  &\times (-1)^{k+q_2-p-q}c_i^{q-q_1+a_1} s_i^{p+1-a_1} c_{i+1}^{l-q_2+p+a_2+2-a_1} s_{i+1}^{k-p-a_2+a_1} c_{i+2}^{q_2+q_1-p-q-a_2-1}s_{i+2}^{a_2}
%  \nn\\
% %
% &- \sum\limits_{a_1=0\vee(q_1-q-1)}^{q_1\wedge p}\sum\limits_{a_2=0\vee(l-q_2-q_1+p+1)}^{(k-p+1)\wedge(q_2+q_1-p-q-1)}
% \binom{p}{a_1}\binom{k-p+1}{a_2} \binom{q+1}{q_1-a_1}
% \binom{l-q}{q_2+q_1-p-q-a_2-1}\nn\\
% %
% &\times (-1)^{k+q_2-p-q} c_i^{q+1-q_1+a_1}s_i^{p-a_1} c_{i+1}^{l-q_2+p+a_2-a_1+1} s_{i+1}^{k-p+1-a_2+a_1} c_{i+2}^{q_2+q_1-p-q-a_2-1}s_{i+2}^{a_2}
% \Big].\label{MFormula} \end{align}
\end{proof}

The structure of $\cP(\theta_1, \dots, \theta_n)$ ensures that a least squares estimator can be reconstructed using polynomial optimization. This follows as Lemma~\ref{MomentsFromh} yields that $\eP=P(\hat{h}_1,\ldots,\hat{h}_n)$ is a least squares estimator of $K$, where $(\hat{h}_1,\ldots,\hat{h}_n)$ is the solution of the polynomial optimization problem
\begin{equation}\label{OptProb}
(\hat{h}_1,\ldots,\hat{h}_n)=
\argmin\{f(h_1,\ldots,h_n):(h_1,\ldots,h_n)\in A_n\},
\end{equation}
where the objective function $f:\R^n\rightarrow [0,\infty)$ is defined by
\[
f(h_1,\ldots,h_n)=\sum\limits_{k,l=0}^{N}
\left(\lambda_{kl}(K)-\sum\limits_{i=1}^n \sum_{s=0}^{k+l} \sum\limits_{q_1=0}^{s+1}\sum\limits_{q_2=0}^{s+2-q_1} L_{kl}(i,s,q_1,q_2)h_i^{q_1}h_{i+1}^{q_2}h_{i+2}^{s+2-q_1-q_2}\right)^2
\]
and the feasible set $A_n$ is the set of vectors $(h_1,\ldots,h_n)\in(-\infty,\infty)^n$ which fulfil the inequalities
\begin{align*}
&  0\leq h_{i-1}(s_{i+1}c_i-c_{i+1}s_i) - h_i(s_{i+1}c_{i-1}-c_{i+1}s_{i-1})+h_{i+1}(s_{i}c_{i-1}-c_{i}s_{i-1}),\\
&0\leq \frac{1}{c_i s_{i+1}-s_i c_{i+1}}
 (h_i s_{i+1}-h_{i+1} s_i)\leq 1, \\
& 0\leq \frac{1}{c_i s_{i+1}-s_i c_{i+1}}(h_{i+1}c_i - h_i c_{i+1} )\leq 1
\end{align*}
for $1\leq i\leq n$. Algorithms for obtaining approximate solutions of polynomial optimization problems like \eqref{OptProb} are an active field of research. We mention the software GloptiPoly, see \cite{Lasserre.2001}, which is recommended for small-scale problems. Another possible choice for a problem like \eqref{OptProb} is the software SparsePop, see \cite{Waki.2008}, which is designed for problems with a sparse structure. However, these specialized approaches become very costly as the number of variables increases. For our problem \eqref{OptProb} we provide some reconstruction examples in Chapter \ref{Chapter:Implementation}
which are obtained by using Matlab's general fmincon routine with optimization by interior point methods.

%%%%%%%%%%%%%%%%%%%%%%%%%%%%%%%%%%%%%%%%%%%%%%%%%%%%%%%%%%%%%%%%%%%%%%%%%%%%%%%%%%%

\subsection{Convergence of the reconstruction algorithm}
In this section, we use the stability result Theorem~\ref{ThmEstimateLegendre} to show that the output polygon of the reconstruction algorithm described in the previous section converges to $K$ in the Nikodym distance when the number $n$ of outer normals of the polygon and the number $N$ of moments increase.

\begin{lem}\label{SymDifEstimate}
Let $K \subset [0,1]^2$ be a convex body, $0 \leq \theta_1 < \dots < \theta_n < 2\pi$ and $\theta_{n+1}=\theta_1$. Assume that condition $\eqref{condition_u}$ is satisfied. Then
\begin{align*}
\delta_N\big(P(h_K(\theta_1),\ldots,h_K(\theta_n)),K\big) 
&\leq
\sqrt{2} \max\limits_{1\leq i\leq n}\tan\left(\frac{\theta_{i+1}-\theta_i}{2}\right).
\end{align*}
\end{lem}
%-----------------------------------------------------------------
\begin{proof}
Choose $x_1,\ldots,x_n\in \partial K$ such that $(\cos(\theta_i), \sin(\theta_i))^{\top}$ is an outer normal of $K$ at $x_i$. Let \[P_{in}:=\conv\{x_1,\ldots,x_n\}.\] Note that $P_{in}\subset K$. Recall that the vertices of $P(h_K(\theta_1),\ldots,h_K(\theta_n))$ are denoted by $v_1,\ldots,v_n$ and let $T_i:=\conv\{x_{i},x_{i+1},v_i\}$, $c_i:=\|x_{i+1}-x_i\|$ and $\gamma_i:=\pi-(\theta_{i+1}-\theta_{i})$.
Then we clearly have
\begin{equation}\label{Est1}
P(h_K(\theta_1),\ldots,h_K(\theta_n))\setminus \intr P_{in}= \bigcup\limits_{i=1}^{n} T_i,
\end{equation}
see Figure \ref{FigureSymDifEstimate}.
\begin{figure}[h!]
  \centering
  \includegraphics[width=5cm]{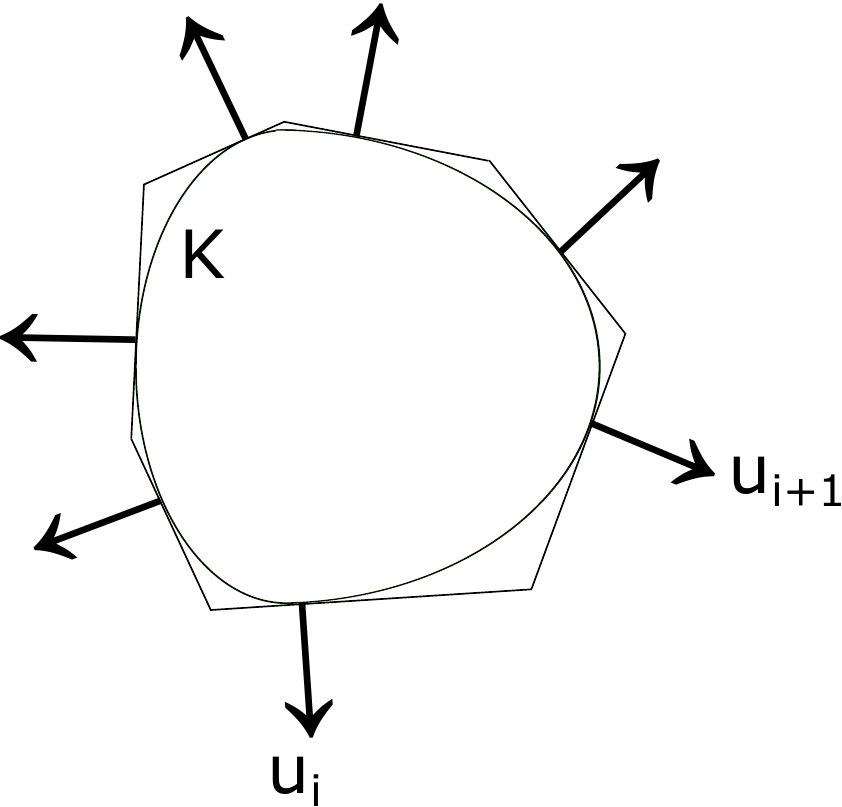}
      \includegraphics[width=5cm]{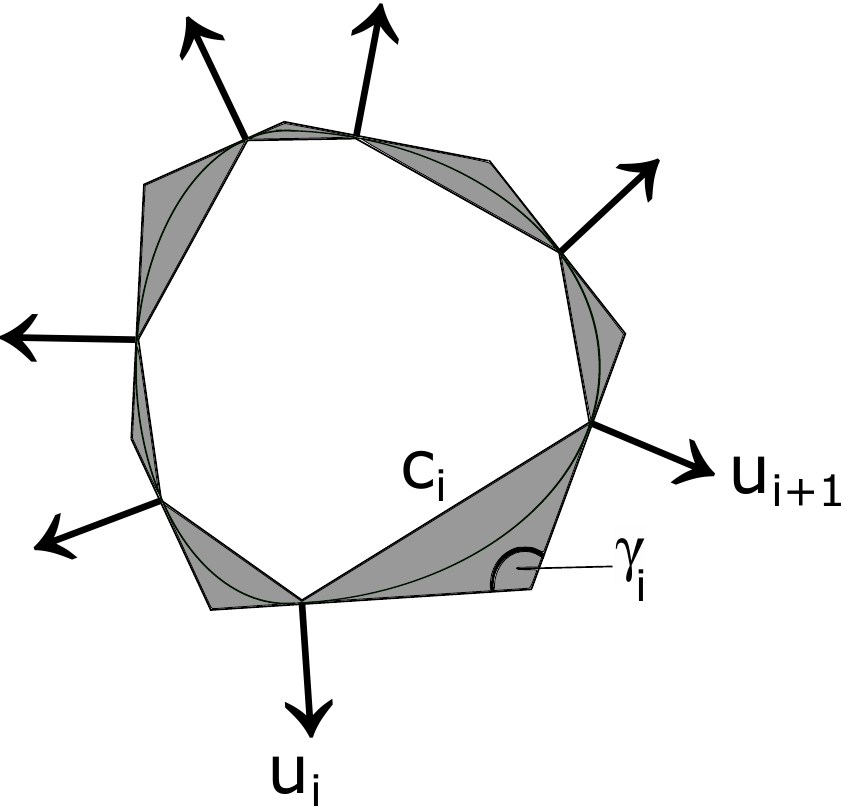}
\caption{On the left, $K$ and the polytope $P(h_K(u_1),\ldots,h_K(u_n))$. On the right, $P(h_K(u_1),\ldots,h_K(u_n))\setminus P_{in}$ coloured in grey.  }\label{FigureSymDifEstimate}
\end{figure}
Observe that the area of a triangle where one angle and the length of the side opposite to the angle are prescribed is maximal if the remaining angles are equal. Thus,
\begin{equation}\label{Est2}V_2(T_i)\leq \frac{1}{4}c_i^2\cot(\gamma_i/2)=\frac{1}{4}c_i^2\tan\left(\frac{\theta_{i+1}-\theta_i}{2}\right). \end{equation}
Equations~\eqref{Est1} and \eqref{Est2} imply that
\begin{align*}%\label{ineq}
\delta_N\big(P(h_K(\theta_1),\ldots,h_K(\theta_n)),K\big) &\leq
V_2(P(h_K(\theta_1),\ldots,h_K(\theta_n))\setminus P_{in}) \nn \\
&\leq\frac{1}{4} \max\limits_{1\leq i\leq n}\tan\left(\frac{\theta_{i+1}-\theta_i}{2} \right) \sum\limits_{i=1}^{n} c_i^2,
\end{align*}
and since $c_i^2/2 \leq c_i / \sqrt{2}$ and $\sum_{i=1}^n c_i = 2 V_1(P_{in})$, we arrive at
\begin{equation*}
\delta_N\big(P(h_K(\theta_1),\ldots,h_K(\theta_n)),K\big) 
\leq
\frac{1}{\sqrt{2}} V_1(P_{in}) \max\limits_{1\leq i\leq n}\tan\left(\frac{\theta_{i+1}-\theta_i}{2}\right). 
\end{equation*}
The monotonicity of intrinsic volumes with respect to set inclusion then yields the assertion.
\end{proof}

%---------------------------------------------------------------------------------------

%%%%%%%%%%%%%%%%%%%%%%%%%%%%%%%%%%%%%%%%%%%%%%%%%%%%%%%%%%%%%%%%%%%%

\begin{thm}\label{Stability2}
Let $K \subset [0,1]^2$ be a convex body, $0 \leq \theta_1 < \dots < \theta_n < 2 \pi$, $\theta_{n+1}=\theta_1$ and assume that $0,\frac{\pi}{2},\pi,\frac{3\pi}{2} \in \{\theta_1,\ldots,\theta_n\}$. Any least squares estimator $\eP$ of $K$ on $\cPt$ satisfies that
\[\delta_N(\eP,K)\leq \sqrt{2}\max\limits_{1\leq i\leq n}\tan\left(\frac{\theta_{i+1}-\theta_i}{2}\right)+ \frac{a_1}{N+1},\]
where $a_1>0$ is a constant.
\end{thm}
%%%%%%%%%%%%%%%%%%%%%%%%%%%%%%%%%%%%%%%%%%%%%%%%%%%%%%%%%%%%%%%%%%%

\begin{proof}
Since $0,\frac{\pi}{2},\pi,\frac{3\pi}{2} \in \{\theta_1,\ldots,\theta_n\}$ and $K\subset [0,1]^2$ it follows that
  \[
  P(h_K(\theta_1),\ldots,h_K(\theta_n))\subset [0,1]^2
  \]
and thus $ P(h_K(\theta_1),\ldots,h_K(\theta_n))\in \cPt$. Then the definition of $\eP$ and Parseval's identity yield that
\begin{align*}
\sum_{i,j=0}^N (\lambda_{ij}(K) - \lambda_{ij}(\eP))^2 &\leq  \sum_{i,j=0}^N [\lambda_{ij}(K) - \lambda_{ij}(P(h_K(\theta_1),\dots, h_K(\theta_n)))]^2
\\
&\leq
\delta_N(P(h_K(\theta_1),\dots, h_K(\theta_n)),K).
\end{align*}
Thus, an application of Lemma~\ref{SymDifEstimate} implies that
\begin{equation}\label{upperbound}
\sum_{i,j=0}^N (\lambda_{ij}(K) - \lambda_{ij}(\eP))^2 \leq  \sqrt{2}\, \max\limits_{1\leq i\leq n}\tan\left(\frac{\theta_{i+1}-\theta_i}{2}\right).
 \end{equation} 
Then the result follows from Theorem~\ref{ThmEstimateLegendre} and Remark~\ref{NikodymExt}.
\end{proof}

\begin{rem}
If we choose $n= 4m$ for some $m\in \N$ and equidistant angles $\theta_i:= 2\pi \left(\frac{i-1}{n}\right)$ for $\quad 1\leq i\leq n$, then  $0,\frac{\pi}{2},\pi,\frac{3\pi}{2} \in \{\theta_1,\ldots,\theta_n\}$
and we obtain
\[\sqrt{2} \max\limits_{1\leq i\leq n}\tan\left(\frac{\theta_{i+1}-\theta_i}{2}\right) \approx \frac{\sqrt{2}\pi}{n} \approx\begin{cases} 0.05,& n=100,\\
0.005,& n=1000,\\0.0025,& n=2000.
\end{cases}\]
\end{rem}

%--------------------------------------------------------------------
%\begin{proof}[Proof of Theorem~\ref{Stability2}]
% On the space of convex bodies $\cK$ we define the pseudometric $\delta_{\lambda,N}$ by
%  \[\delta_{\lambda,N}(K,L):=\sqrt{\sum\limits_{i,j=0}^N \left(\lambda_{ij}(K)-\lambda_{ij}(L)\right)^2}, \quad K,L \in \cK.\]
% By Theorem~\ref{ThmEstimateLegendre} and Remark~\ref{NikodymExt}, the distance of the estimator $\eP$ to the original convex body $K$ in the symmetric difference metric $\delta_N$ can be bounded by
%  \[
%  \delta_N(\eP,K)\leq \delta_{\lambda,N}(\eP,K)^2 + C \frac{1}{N}.
%  \]
%  From $0,\frac{\pi}{2},\pi,\frac{3\pi}{2} \in \{\theta_1,\ldots,\theta_n\}$ and $K\subset [0,1]^2$ it follows that
%  \[
%  P(h_K(\theta_1),\ldots,h_K(\theta_n))\subset [0,1]^2
%  \]
%  and thus $ P(h_K(\theta_1),\ldots,h_K(\theta_n))\in \cPt$.
%  Therefore, by the definition of $\eP$ we have
%  \[
%  \delta_{\lambda,N}(\eP,K)\leq   \delta_{\lambda,N}\big(P(h_K(\theta_1),\ldots,h_K(\theta_n)),K\big)
%  \]
%  and by Parseval's identity and Lemma \ref{SymDifEstimate}, it follows that
% \begin{align*} \delta_{\lambda,N}\big(P(h_K(\theta_1),\ldots,h_K(\theta_n)),K\big)^2
%  &\leq \delta_N\big(P(h_K(\theta_1),\ldots,h_K(\theta_n)),K\big)\\
%  %
%  &\leq \frac{1}{\sqrt{2}}\,V_1(K) \max\limits_{1\leq i\leq n-1}\tan\left(\frac{\theta_{i+1}-\theta_i}{2}\right)\\
%  %
%  &\leq \sqrt{2}\, \max\limits_{1\leq i\leq n-1}\tan\left(\frac{\theta_{i+1}-\theta_i}{2}\right).
%  \end{align*}
%\end{proof}

%----------------------------------------------------------------------
In the following, we write $\theta_{(1)}, \dots, \theta_{(n)}$ for a permutation of $\theta_i \in [0,2\pi), 1 \leq i \leq n$ satisfying $\theta_{(1)} \leq \dots \leq \theta_{(n)}$. From Theorem~\ref{Stability2}, we then obtain Corollary~\ref{cor_convergence}.

\begin{cor}\label{cor_convergence}
Let $K \subset [0,1]^2$ be a convex body and let $(\theta_i)_{i \in \N}$ be a dense sequence in $[0, 2\pi)$ such that $\theta_i \neq \theta_j$ for $i \neq j$ and $(\theta_1, \theta_2, \theta_3, \theta_4)=(0, \frac{\pi}{2}, \pi, \frac{3 \pi}{2})$. For $n,N \in \N $, let $\eP$ be a least squares estimator of $K$ with respect to the $(N+1)^2$ first Legendre moments on the space $\cP(\theta_{(1)}, \dots, \theta_{(n)})$. Then 
\begin{equation*}
\delta_N(K, \eP) \rightarrow 0 \text{ for } n,N \rightarrow \infty.
\end{equation*}
\end{cor}

%%%%%%%%%%%%%%%%%%%%%%%%%%%%%%%%%%%%%%%%%%%%%%%%%%%%%%%%%%%%%%%%%%%%%

\subsection{Reconstruction from noisy measurements}\label{SecNoise}
The reconstruction algorithm described in Section~\ref{Sec_LSQLegendre} requires knowledge of exact Legendre moments of a convex body. The reconstruction algorithm can be modified such that it allows for noisy measurements of Legendre moments. Let $N \in \N_0$, and assume that $K \subset [0,1]^2$ is a convex body where measurements of the first $(N+1)^2$ Legendre moments are known. To include noise, we assume that the measurements are of the form
\begin{equation} \label{measurements}
\tilde{\lambda}_{kl}(K) = \lambda_{kl}(K) + \epsilon_{Nkl}
\end{equation}
for $k,l=0, \dots, N$, where $\epsilon_{Nkl}$, $k,l =0, \dots, N$ are random variables with zero means and finite variances bounded by a constant $\sigma_N^2$. Let $0 \leq \theta_1 < \dots < \theta_n < 2\pi$ satisfy condition~\eqref{condition_u}. Any polygon $\ePn \in \cPt$ satisfying
\begin{equation*}
\ePn = \argmin \left\{ \sum\limits_{k,l=0}^N (\tilde{\lambda}_{kl}(K)-\lambda_{kl}(P))^2:\quad P\in\cP(\theta_1,\ldots,\theta_n) \right\}
\end{equation*}
is called a least squares estimator of $K$ with respect to the measurements~\eqref{measurements} on the space $\cPt$. As the set $\cPt$ is closed in the Hausdorff metric, Blaschke's selection theorem ensures the existence of a least squares estimator. 

As in Section~\ref{Sec_LSQLegendre}, a least squares estimator can be found using polynomial optimization. Let $(\tilde{h}_1, \dots, \tilde{h}_n)$ be a solution to the polynomial optimization problem \eqref{OptProb} with the Legendre moments $\lambda_{kl}(K)$ of $K$ replaced by the measurements $\tilde{\lambda}_{kl}(K)$ of the Legendre moments in the objective function $f$. Then $P(\tilde{h}_1, \dots, \tilde{h}_n) \in \cPt$ is a least squares estimator of $K$ with respect to the measurements \eqref{measurements}. 
  
Now, let $\sol$ denote the random set of least squares estimators of $K$ with respect to the measurements \eqref{measurements} on the space $\cPt$. When the noise variables are defined on a complete probability space $(\Omega, \cF, \P)$, it follows by arguments as in \cite[p. 27]{Kousholt2016} (see also \cite[App. C]{Pollard1984}) that $\sup_{P \in \sol} \delta_N(K,P)$ is $(\cF, \cB(\R))$-measurable. We can then formulate the following theorem, which ensures consistency of the reconstruction algorithm under certain assumptions on the variances of the noise variables.

%\begin{lem} \label{LemmaBound}
%Assume that $e_1, -e_1, e_2, -e_2 \in \{u_1, \dots, u_n\}$, and let $\ePn$ be a least squares estimator of $K$ based on the measurements~\eqref{measurements}. Then
%\begin{align} \label{BoundNoise}
%\delta_N(K, \ePn)
%\leq 16 \sqrt{2} \max_{1 \leq i \leq n-1}\tan\left(\frac{\theta_{i+1} - \theta_i}{2}\right) +  C\frac{1}{N+1} + 20\sum_{k,l=0}^N \epsilon_{kl}^2.
%\end{align}
%\end{lem}
%
%\begin{proof} Let $\eP \in \cP(\theta_1, \dots, theta_n)$ denote the least squares estimator of $K$ based on the exact Legendre moments. By using the inequality $\lvert x + y \rvert^2 \leq 4(\lvert x^2 \rvert + \lvert y \rvert^2 )$ for $x,y \in \R$ and the properties of $\ePn$ and $\eP$, we obtain that
%\begin{align*}
%& \sum_{k,l=0}^N (\lambda_{kl}(K)-\lambda_{kl}(\ePn))^2 
%\leq 
%4 \sum_{k,l=0}^N \left((\tilde{\lambda}_{kl}(K)-\lambda_{kl}(\ePn))^2 + \epsilon_{kl}^2 \right)
%\\
%& \leq 
%16 \sum_{k,l=0}^N (\tilde{\lambda}_{kl}(K)-\lambda_{kl}(K))^2 + 16 \sum_{k,l=0}^N (\lambda_{kl}(K)-\lambda_{kl}(\eP))^2 +  4\sum_{k,l=0}^N \epsilon_{kl}^2
%\\
%&\leq
%16 \delta_{\lambda,N}(K,\eP)^2 + 20\sum_{k,l=0}^N \epsilon_{kl}^2.
%\end{align*}
%Using the upper bound on $\delta_{\lambda,N}(K,\eP)$ derived in the proof of Theorem~\ref{Stability2}, we arrive at
%\begin{equation*}
%\sum_{k,l=0}^N (\lambda_{kl}(K)-\lambda_{kl}(\ePn))^2 \leq 16 \sqrt{2} \max_{1 \leq i \leq n-1}\tan\left(\frac{\theta_{i+1} - \theta_i}{2}\right) + 20\sum_{k,l=0}^N \epsilon_{kl}^2.
%\end{equation*}
%Then inequality~\eqref{BoundNoise} follows by Theorem~\ref{ThmEstimateLegendre}.
%\end{proof}

\begin{thm}\label{ThmNoise}
Let $(\theta_i)_{i \in \N}$ be a dense sequence in $[0,2\pi)$ such that $\theta_i \neq \theta_j$ for $i \neq j$ and $(\theta_1, \theta_2, \theta_3, \theta_4)=(0,\frac{\pi}{2}, \pi, \frac{3\pi}{2})$.
\begin{enumerate}[(i)]
\item If $\sigma_N^2 = \cO(\frac{1}{N^{2+ \varepsilon}})$ for some $\varepsilon > 0$, then $\sup_{P \in \sol} \delta_N(K,P) \rightarrow 0$ in mean and in probability for $n,N \rightarrow \infty$.
\item If $\sigma_N^2 = \cO(\frac{1}{N^{3+\varepsilon}})$ for some $\varepsilon > 0$, then $\sup_{P \in \sol} \delta_N(K, P) \rightarrow 0$ almost surely for $n,N \rightarrow \infty$.
\end{enumerate}
 
\end{thm}

\begin{proof}
Let $\theta_{(1)} < \dots < \theta_{(n)}$ be an ordering of $\theta_1, \dots, \theta_n$ and $\theta_{(n+1)}:=\theta_1$. In the notation, we suppress that the ordering depends on $n$.
Let $\sol$ be the set of least squares estimators of $K$ with respect to the noisy measurements $\tilde{\lambda}_{kl}(K)$ on the space $\cP(\theta_{(1)},\ldots,\theta_{(n)})$.
Furthermore, let $P \in \sol$ and let $\eP \in \cP(\theta_{(1)}, \dots, \theta_{(n)})$ denote a least squares estimator of $K$ with respect to the exact Legendre moments. We use the inequality $(x + y)^2 \leq 2(x^2 + y^2)$ for $x,y \in \R$ to derive the first and third of the subsequent inequalities and the fact that $P$ is a least squares estimator with respect to the noisy measurements $\tilde{\lambda}_{kl}(K)$ to derive the second inequality. Thus, we obtain that
\begin{align*}
& \sum_{k,l=0}^N (\lambda_{kl}(K)-\lambda_{kl}(P))^2 
\leq 
2 \sum_{k,l=0}^N \left((\tilde{\lambda}_{kl}(K)-\lambda_{kl}(P))^2 + \epsilon_{Nkl}^2 \right)
\\
& \leq 
2 \sum_{k,l=0}^N (\tilde{\lambda}_{kl}(K)-\lambda_{kl}(\eP))^2 +  2\sum_{k,l=0}^N \epsilon_{Nkl}^2
\\
&\leq
4 \sum_{k,l=0}^N (\lambda_{kl}(K)-\lambda_{kl}(\eP))^2 + 6\sum_{k,l=0}^N \epsilon_{Nkl}^2.
\end{align*}
Using the upper bound \eqref{upperbound} on $\sum_{k,l=0}^N (\lambda_{kl}(K)-\lambda_{kl}(\eP))^2$  derived in the proof of Theorem~\ref{Stability2}, we arrive at
\begin{equation*}
\sum_{k,l=0}^N (\lambda_{kl}(K)-\lambda_{kl}(P))^2 \leq 4 \sqrt{2} \max_{1 \leq i \leq n} \bigg\lvert \tan\left(\frac{\theta_{(i)} - \theta_{(i+1)}}{2}\right)\bigg\rvert+ 6\sum_{k,l=0}^N \epsilon_{Nkl}^2.
\end{equation*}
 Then it follows from Theorem~\ref{ThmEstimateLegendre} and Remark~\ref{NikodymExt} that
\begin{equation*}
\sup_{P \in \sol}\delta_N(K, P) \leq 4 \sqrt{2} \max_{1 \leq i \leq n} \bigg \lvert \tan\left(\frac{ \theta_{(i)} - \theta_{(i+1)}}{2}\right) \bigg\rvert + 6\sum_{k,l=0}^N \epsilon_{Nkl}^2 + \frac{a_1}{N+1}.
\end{equation*}
The mean of the sum of the squared error terms are bounded by $(N+1)^2 \sigma_N^2$, and the assumption that $\sigma_N^2 = \cO(\frac{1}{N^{2+ \varepsilon}})$ ensures that $(N+1)^2 \sigma_N^2 \rightarrow 0$ for $N \rightarrow \infty$. As the sequence $(\theta_i)_{i \in \N}$ is dense in $[0, 2\pi)$, we further have that
\begin{equation*}
 \max_{ 1 \leq i \leq n} \bigg \lvert \tan\left(\frac{ \theta_{(i)} - \theta_{(i+1)}}{2}\right) \bigg\rvert \rightarrow 0
\end{equation*}
for $n \rightarrow \infty$. Hence, $\sup_{P \in \sol}\delta_N(K,P) \rightarrow 0$ in mean and in probability for $n,N \rightarrow \infty$.

If $\sigma_N^2=\cO(\frac{1}{N^{3+\varepsilon}})$, then $\sum_{N=0}^\infty (N+1)^2 \sigma_N^2 < \infty$, which ensures that $\sum_{k,l=0}^N \epsilon_{Nkl}^2 \rightarrow 0$ almost surely for $N \rightarrow \infty$. Then, $\sup_{P \in \sol}\delta_N(K,P) \rightarrow 0$ almost surely for $N,n \rightarrow \infty$.

\end{proof}
\section{Implementation and examples of reconstruction}\label{Chapter:Implementation}
We implemented the reconstruction from geometric and Legendre moments, respectively, in Matlab. The code is available at \url{https://gitlab.com/julia.c.schulte/reconstructionfrommoments}. For the optimization we use Matlab's local optimization routine fmincon with interior point optimization and a regular $n$-gon as starting value.
As examples we consider the reconstruction of the following three convex bodies: 
\begin{itemize}
\item the square $K_1:=[0.25,0.75]\times[0,0.5]$, 
\item the half disk $K_2:=\{(x,y)\in\R^2: (x-0.3)^2+(y-0.3)^2\leq 0.3^2\text{ and }y\geq 0.3\}$, 
\item the convex body of constant width $K_3$ which is bounded by the curve 
\[\left\{\begin{pmatrix}0.5+0.05(9\cos(\varphi) +2\cos(2\varphi)-\cos(4\varphi))\\
0.5+0.05(9\sin(\varphi)-2\sin(2\varphi)-\sin(4\varphi))\end{pmatrix}: \varphi\in[0,2\pi]\right\}.\]
\end{itemize}
The third convex body $K_3$ is the positivity set of a polynomial of order 8, see \cite{Rabinowitz.1997}. Thus, by Corollary \ref{SecConsequences} it is uniquely determined in the class $\cK^2$ by its geometric moments (or its Legendre moments) up to order $8$.
We consider the reconstruction from moments up to order $N= 1,3,5$ and with $n=8,32$  sides with normal directions which are equidistant in $[0,2\pi]$.
The calculation of the geometric and Legendre moments for the three bodies was also done in Matlab.
For the reconstruction from noisy moments we add noise $\varepsilon_{Nkl}\sim \cN(0,\sigma_N^2)$ with $\sigma_N=0.01*\min\left(1,N^{-\frac{3.5}{2}}\right)$.
The reconstructions from exact geometric moments are shown in Figures \ref{fig:GeometricNoise0BodyNr1}, \ref{fig:GeometricNoise0BodyNr4} and \ref{fig:GeometricNoise0BodyNr3} and from noisy geometric moments in Figures \ref{fig:GeometricNoise1BodyNr1}, \ref{fig:GeometricNoise1BodyNr4} and \ref{fig:GeometricNoise1BodyNr3}. The reconstructions from exact Legendre moments are displayed in Figures \ref{fig:LegendreNoise0BodyNr1}, \ref{fig:LegendreNoise0BodyNr4} and \ref{fig:LegendreNoise0BodyNr3} and from noisy Legendre moments in Figures \ref{fig:LegendreNoise1BodyNr1}, \ref{fig:LegendreNoise1BodyNr4} and \ref{fig:LegendreNoise1BodyNr3}.  
The reconstructions of $K_1$ and $K_2$ from geometric moments are poor with and without noise whereas the reconstruction of $K_3$ is successful even with noise. 
The reconstruction of all bodies from Legendre moments is successful, though it is clearly visible that the corners are difficult to reconstruct. This is an effect of the optimization procedure which delivers only an approximate solution. The effect of noise is also visible especially when comparing the reconstructions with $N=3$ and $n=32$ in Figures \ref{fig:LegendreNoise0BodyNr1} and \ref{fig:LegendreNoise1BodyNr1} or in Figures \ref{fig:LegendreNoise0BodyNr4} and \ref{fig:LegendreNoise1BodyNr4}. It should be noted that the number $n$ of sides of the reconstructions can be chosen independently of the maximal order $N$ of the available moments. Comparing with the reconstruction from moments of the surface area measure \cite{Kousholt2016} the reconstruction from geometric or Legendre moments has the clear advantage that the number $n$ of sides of the reconstruction can be chosen independently of the maximal order $N$ of the available moments and is not bounded by $2N+1$ as it is the case for the reconstruction from moments of the surface area measure. Especially for smooth convex bodies this leads to a good reconstruction already from very few moments, compare e.g. Figure \ref{fig:LegendreNoise0BodyNr3}.

\newpage
\subsection{Example 1}
\begin{figure}[h!]
	\centering
		\includegraphics[width=1\textwidth]{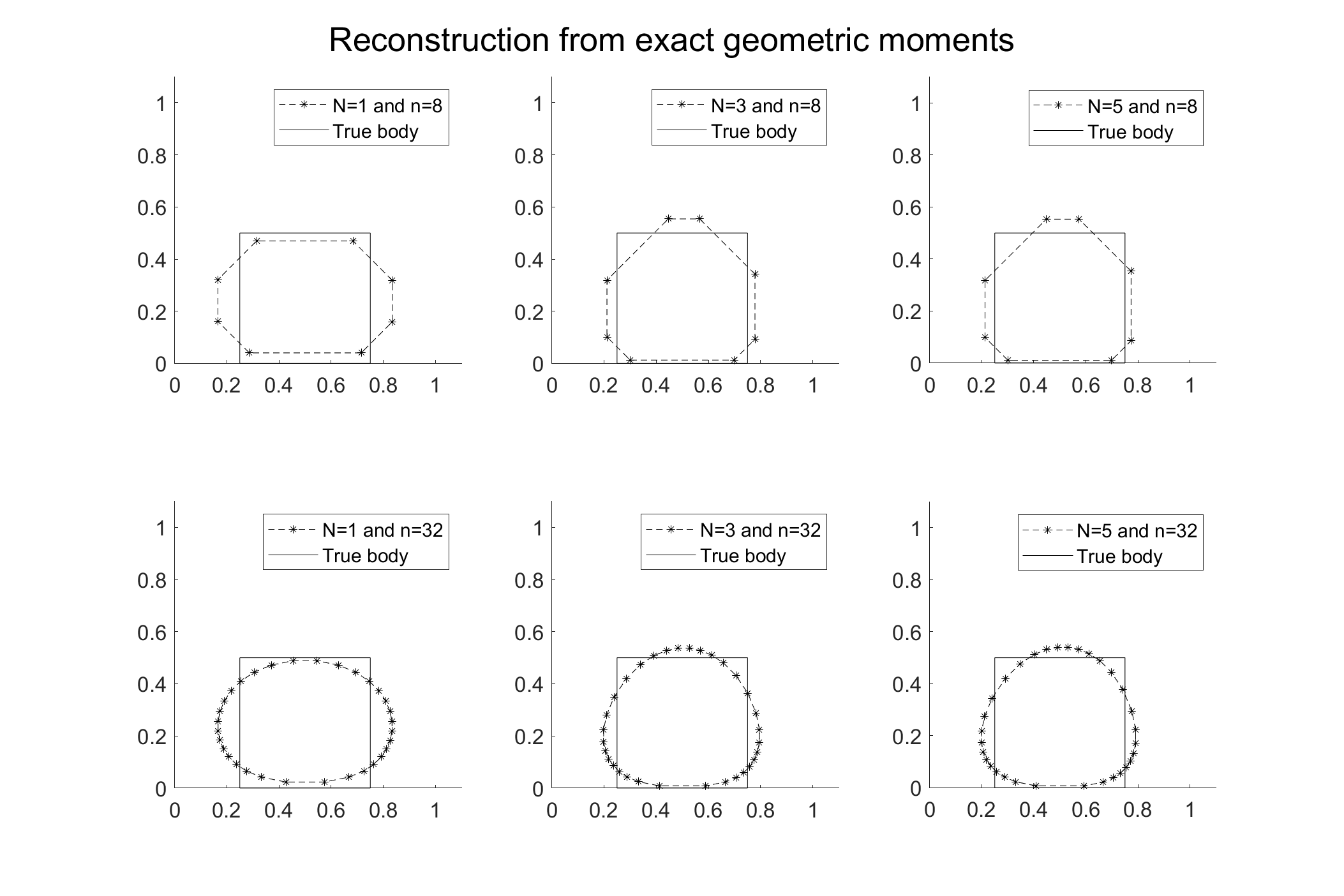}
	\caption{Reconstruction of the body $K_1$ from exact geometric moments.}
			\label{fig:GeometricNoise0BodyNr1}
\end{figure}

\begin{figure}[h!]
	\centering
		\includegraphics[width=1\textwidth]{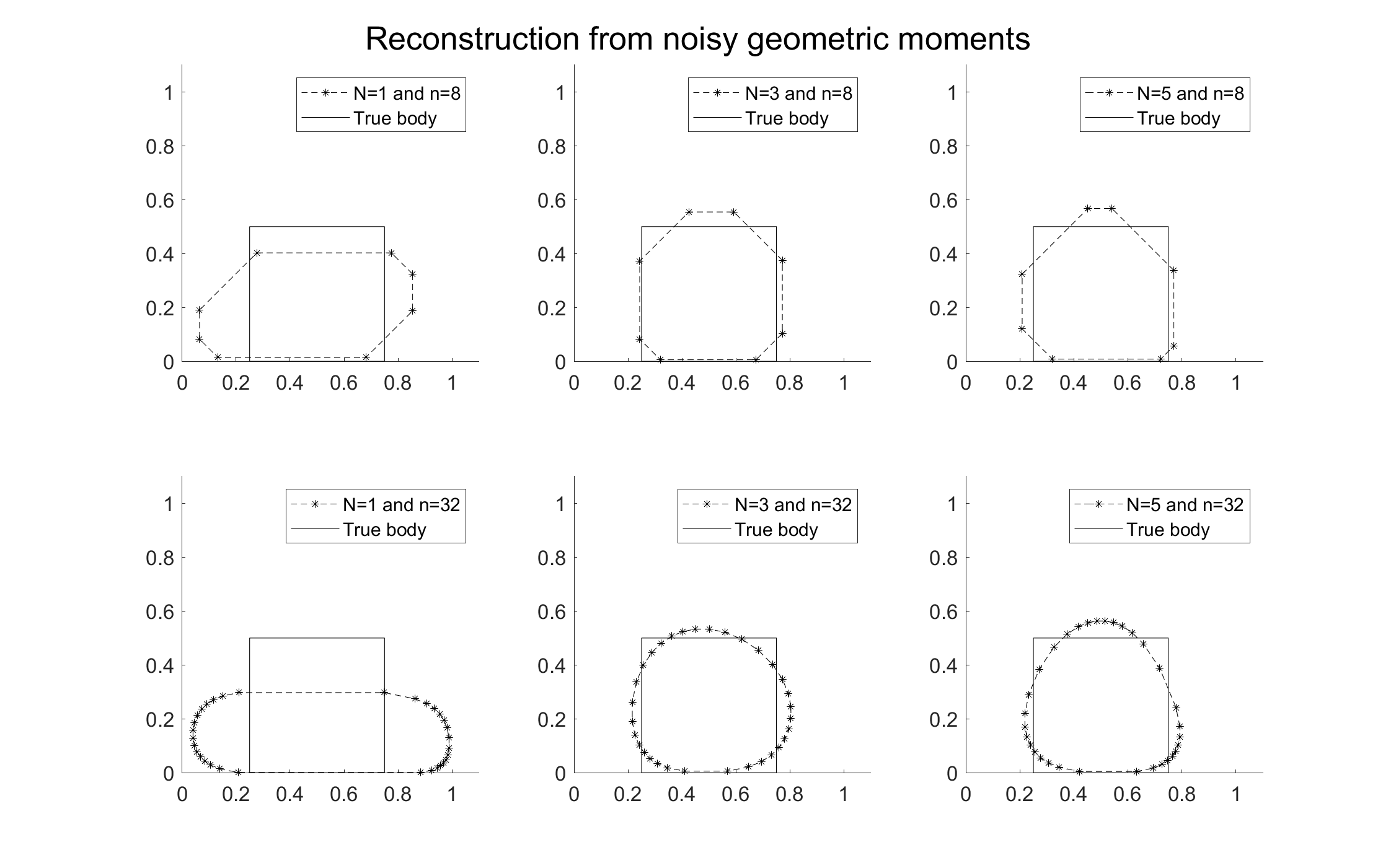}
	\caption{Reconstruction of the body $K_1$ from noisy geometric moments.}
				\label{fig:GeometricNoise1BodyNr1}
\end{figure}
\newpage
\begin{figure}[h!]
	\centering
		\includegraphics[width=1\textwidth]{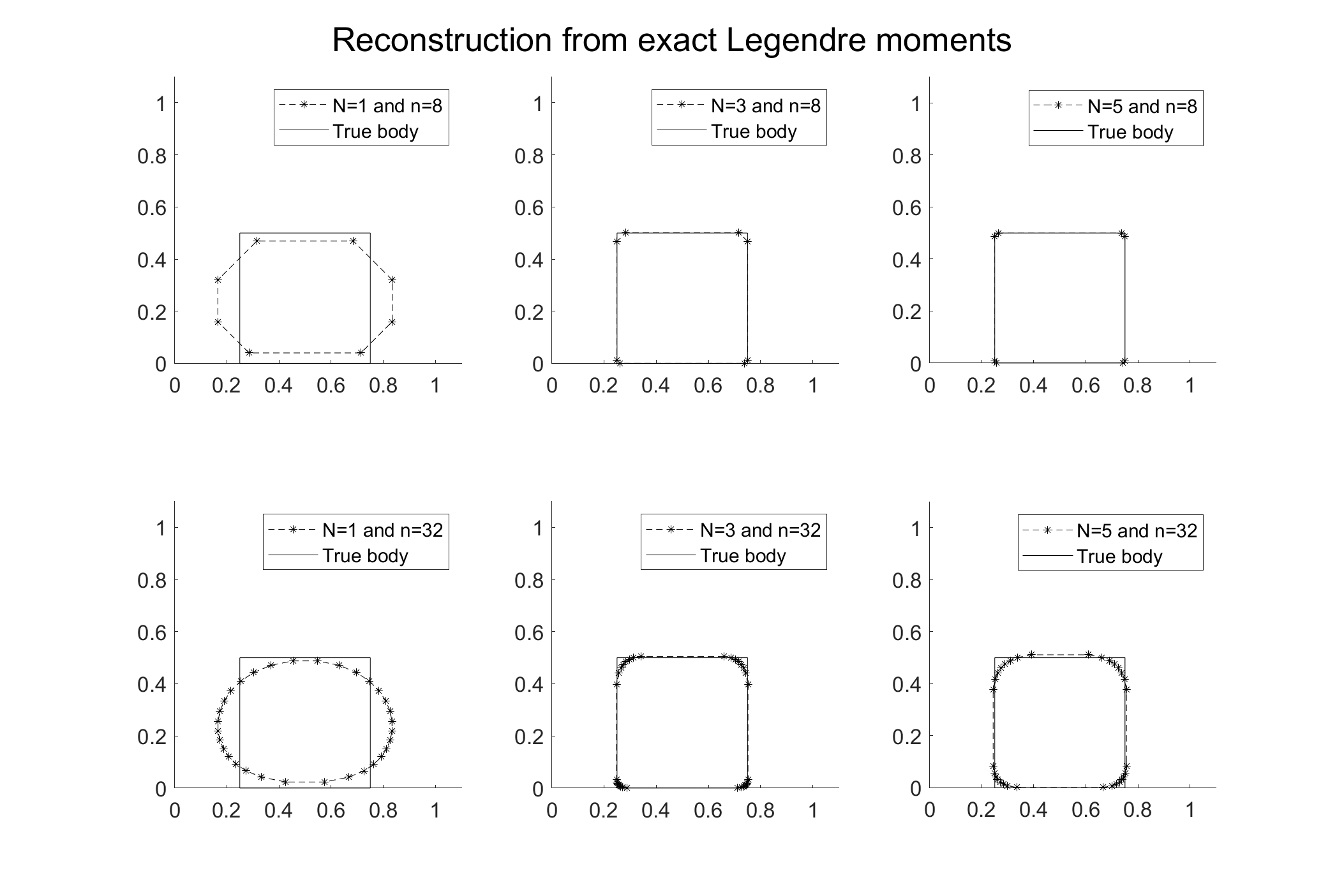}
	\caption{Reconstruction of the body $K_1$ from exact Legendre moments.}
				\label{fig:LegendreNoise0BodyNr1}
\end{figure}
\begin{figure}[h!]
	\centering
		\includegraphics[width=1\textwidth]{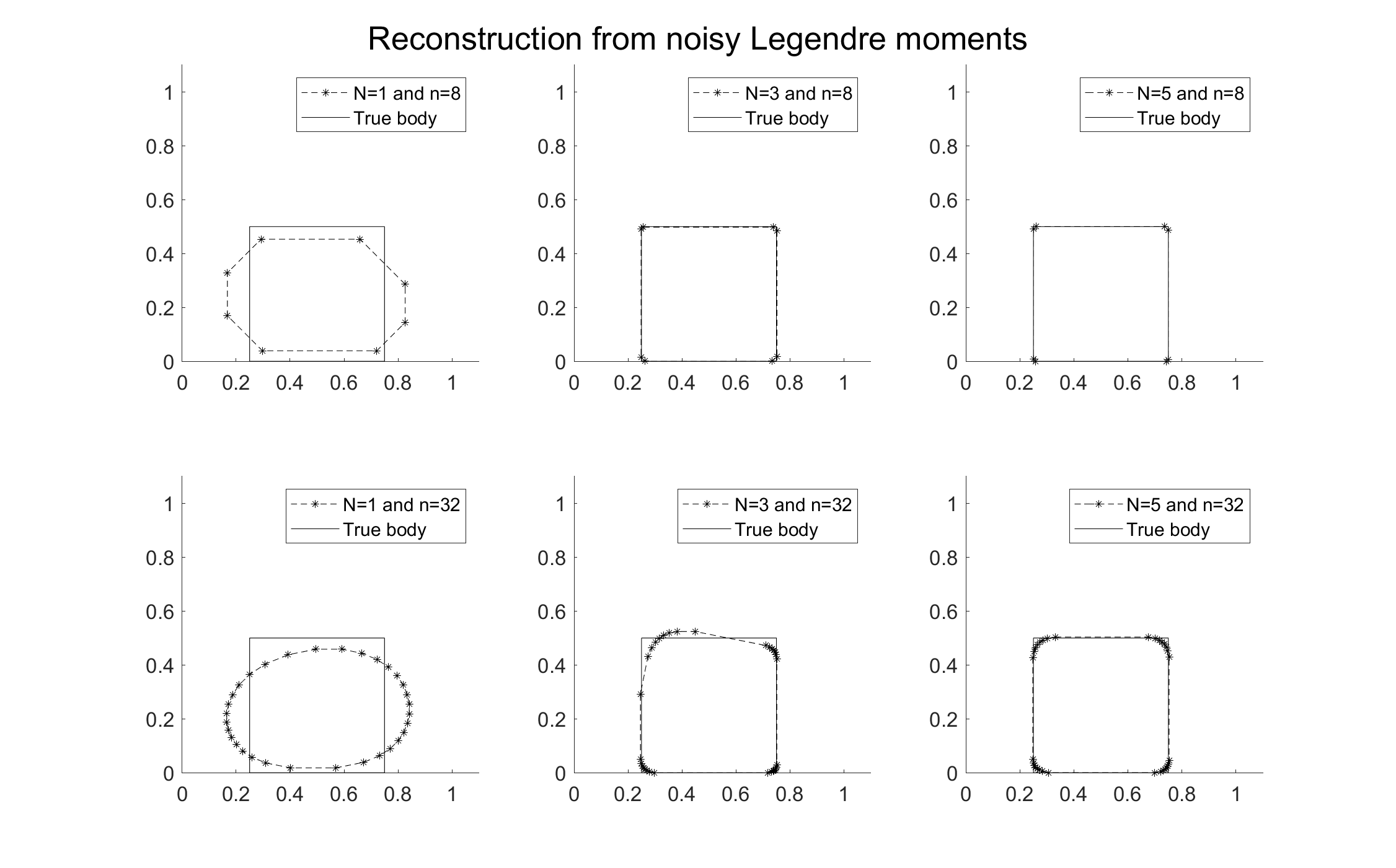}
	\caption{Reconstruction of the body $K_1$ from noisy Legendre moments.}
			\label{fig:LegendreNoise1BodyNr1}
\end{figure}

\newpage
\subsection{Example 2}

\begin{figure}[h!]
	\centering
		\includegraphics[width=1\textwidth]{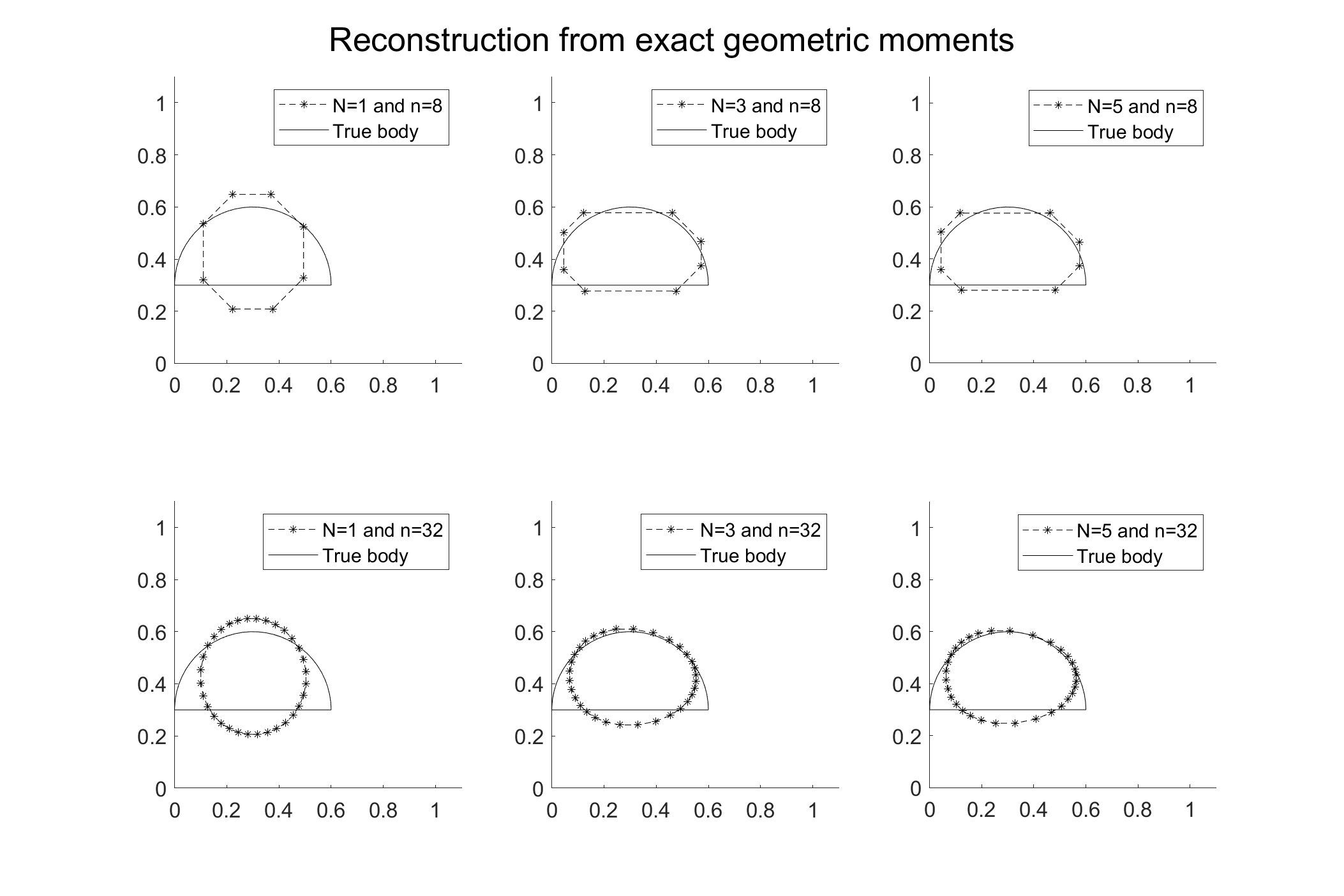}
	\caption{Reconstruction of the body $K_2$ from exact geometric moments.}
					\label{fig:GeometricNoise0BodyNr4}
\end{figure}
\begin{figure}[h!]
	\centering
		\includegraphics[width=1\textwidth]{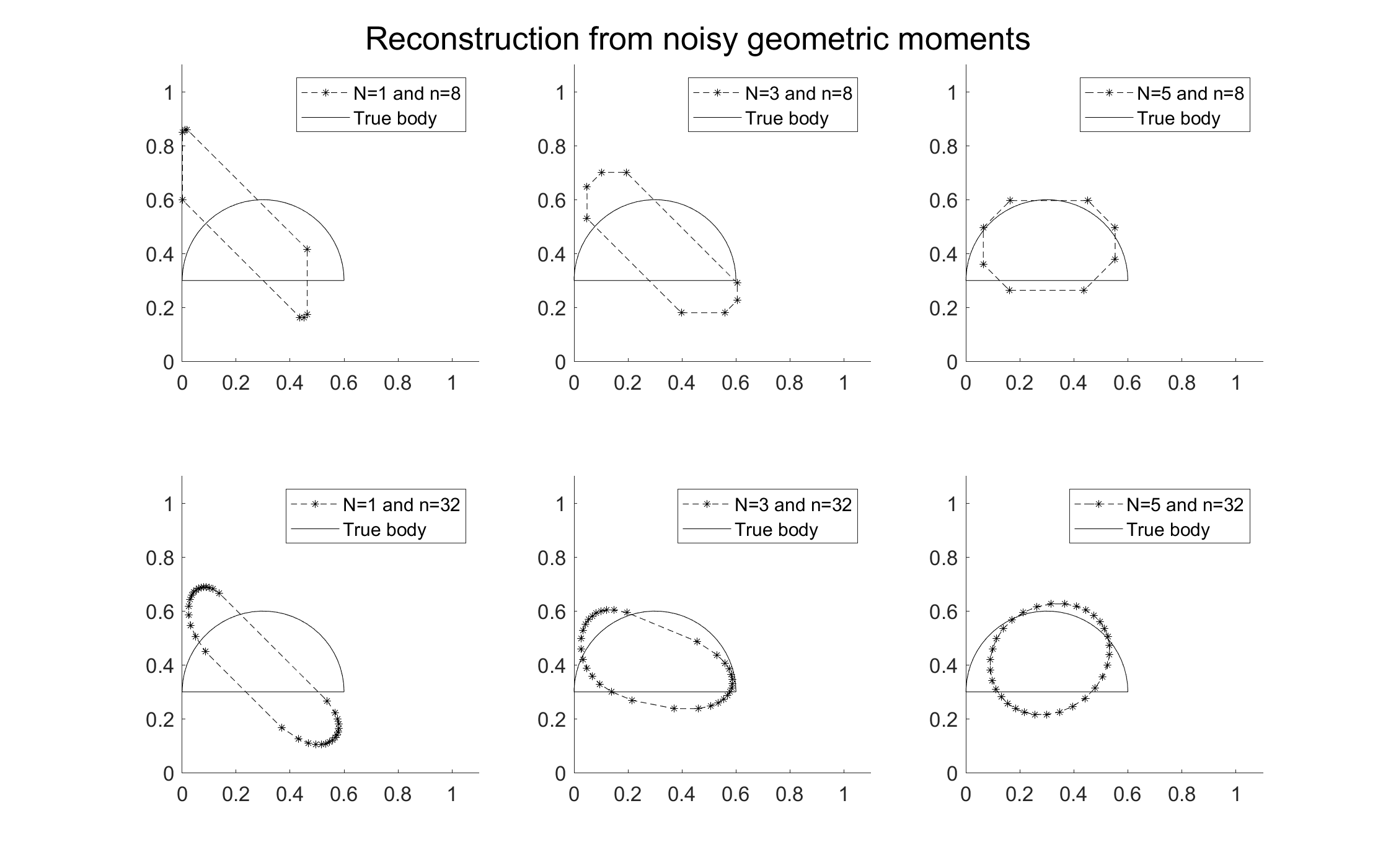}
	\caption{Reconstruction of the body $K_2$ from noisy geometric moments.}
				\label{fig:GeometricNoise1BodyNr4}
\end{figure}
\newpage
\begin{figure}[h!]
	\centering
		\includegraphics[width=1\textwidth]{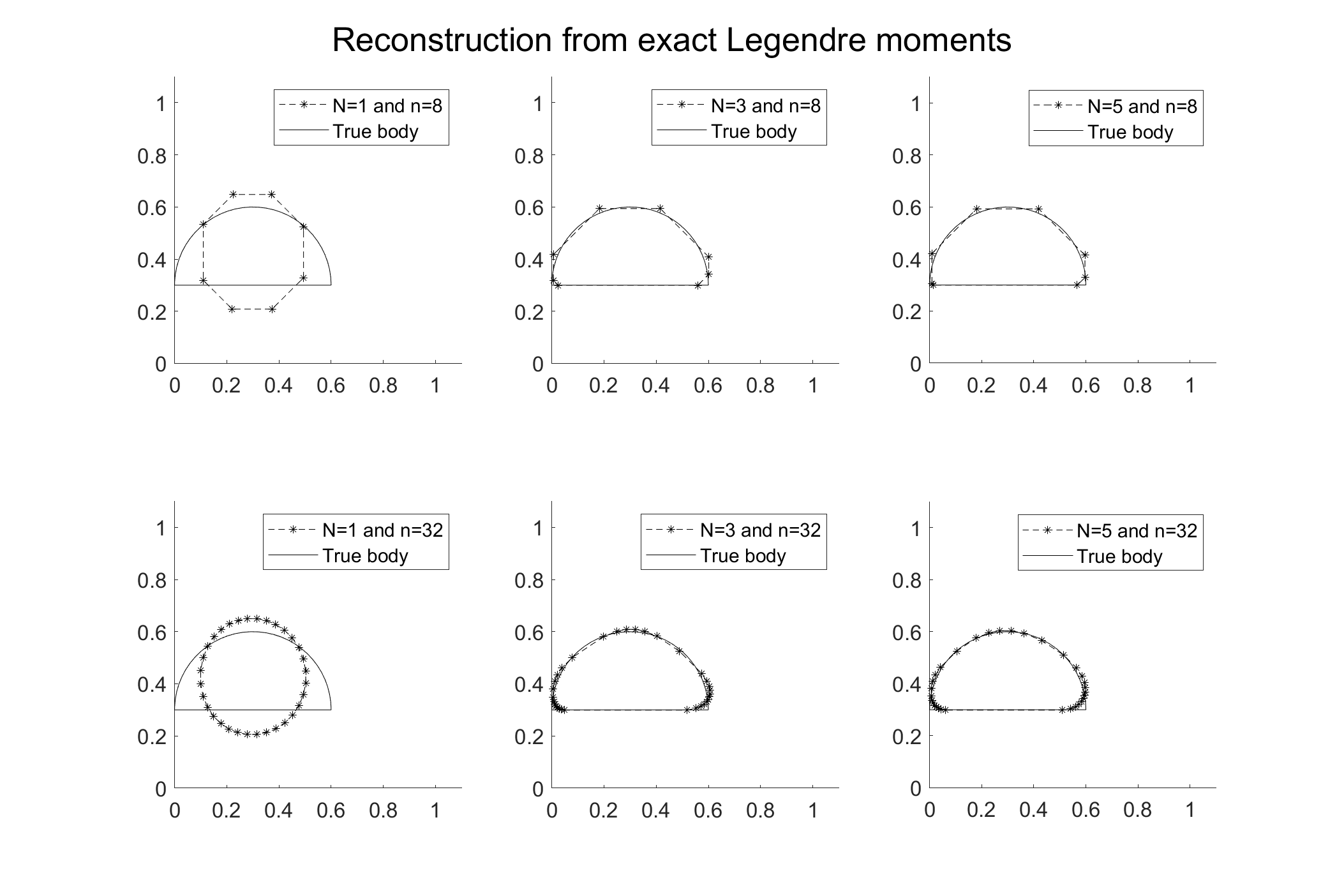}
	\caption{Reconstruction of the body $K_2$ from exact Legendre moments.}
			\label{fig:LegendreNoise0BodyNr4}
\end{figure}
\begin{figure}[h!]
	\centering
		\includegraphics[width=1\textwidth]{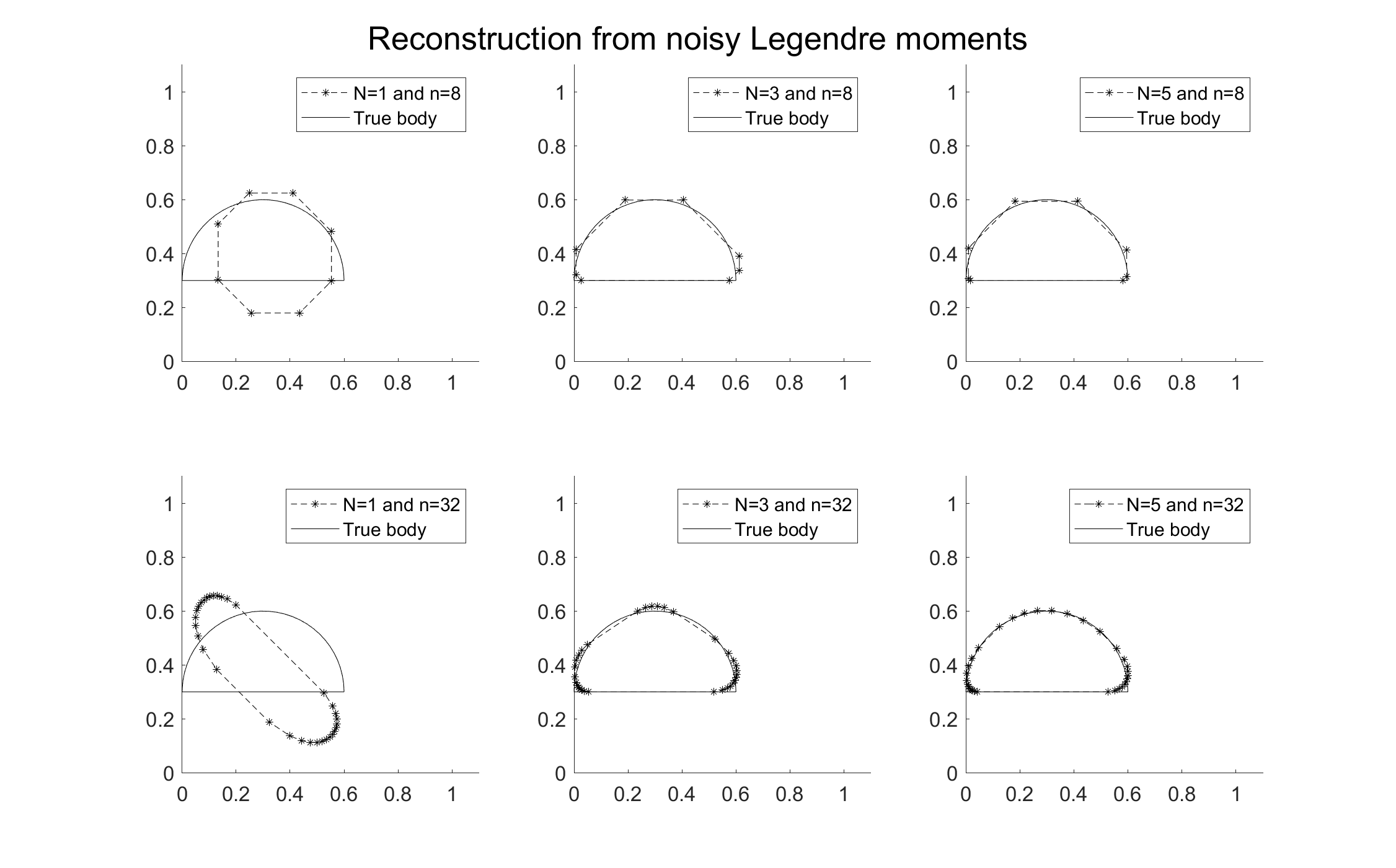}
	\caption{Reconstruction of the body $K_2$ from noisy Legendre moments.}
			\label{fig:LegendreNoise1BodyNr4}
\end{figure}

\newpage
\subsection{Example 3}
\begin{figure}[h!]
	\centering
		\includegraphics[width=1\textwidth]{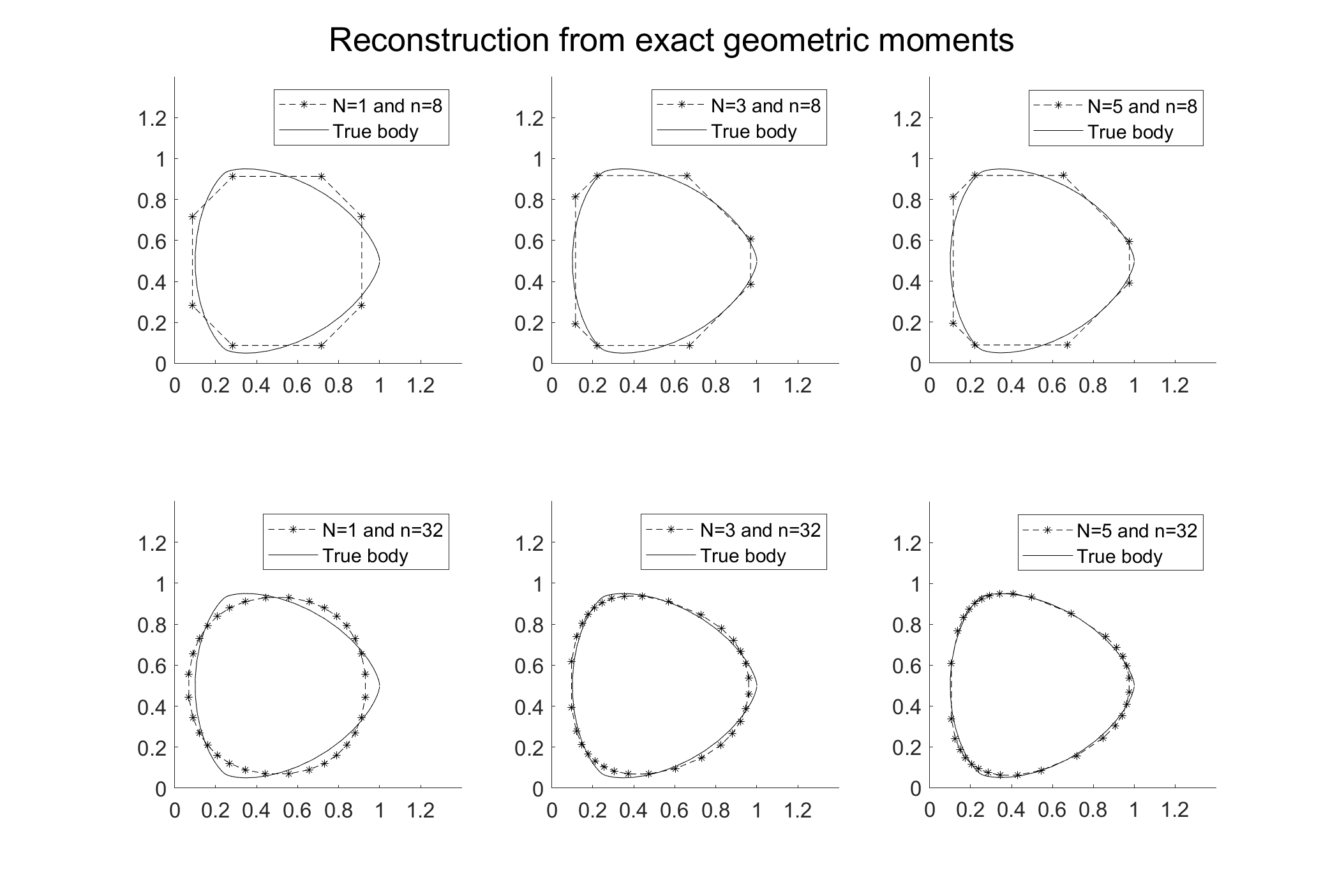}
		\caption{Reconstruction of the body $K_3$ from exact geometric moments.}
					\label{fig:GeometricNoise0BodyNr3}
\end{figure}
\begin{figure}[h!]
	\centering
		\includegraphics[width=1\textwidth]{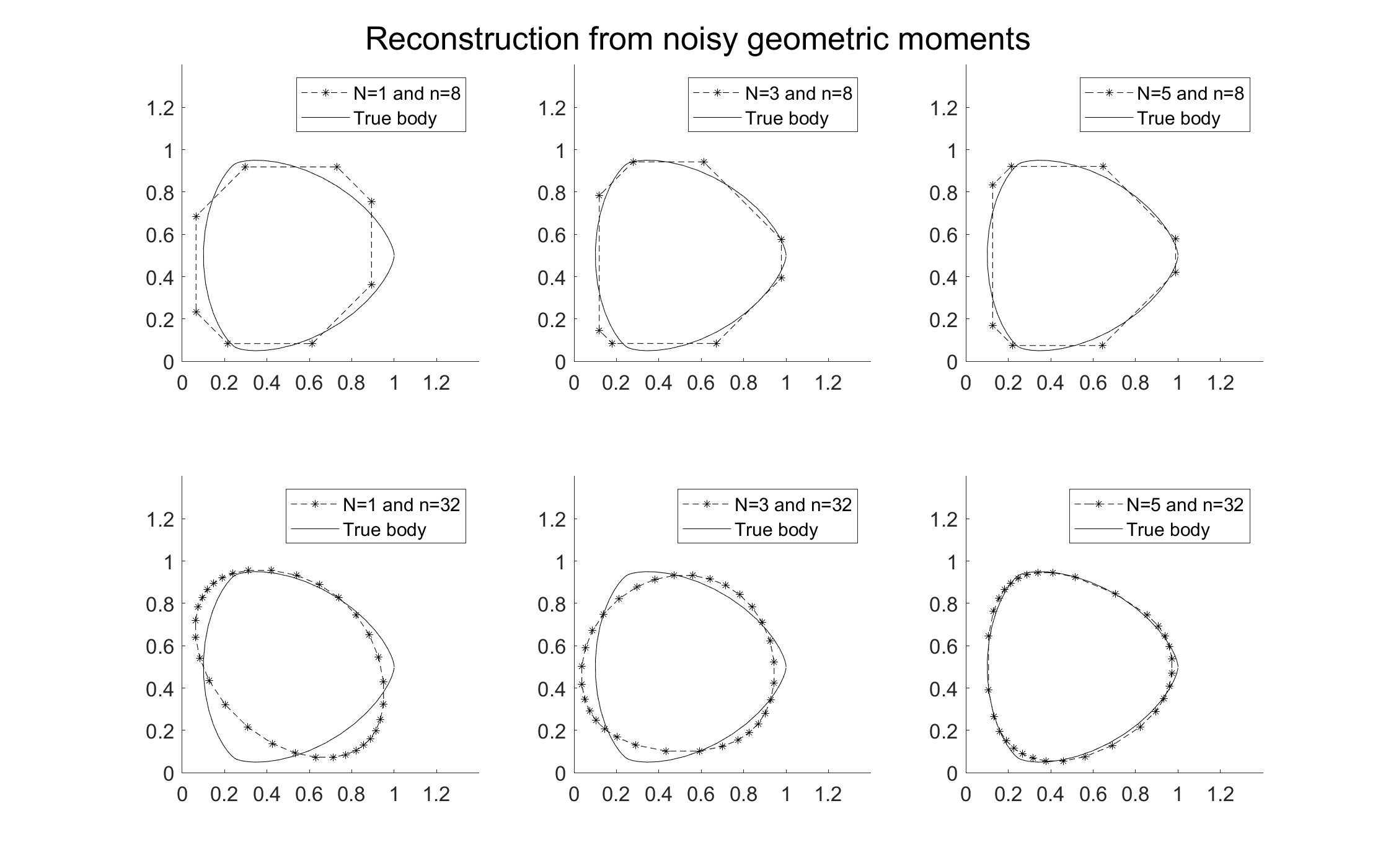}
		\caption{Reconstruction of the body $K_3$ from noisy geometric moments.}
							\label{fig:GeometricNoise1BodyNr3}
\end{figure}
\newpage
\begin{figure}[h!]
	\centering
		\includegraphics[width=1\textwidth]{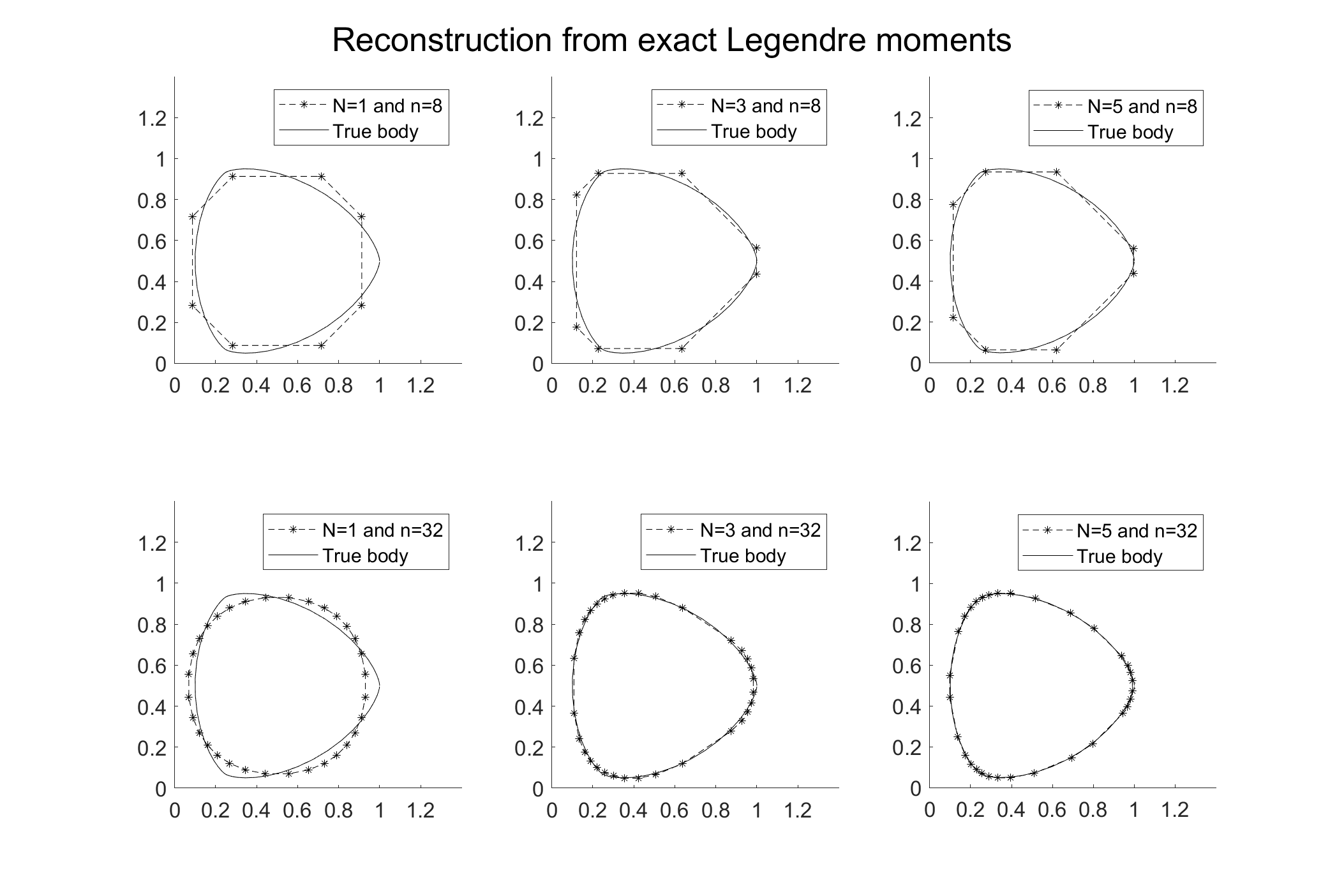}
		\caption{Reconstruction of the body $K_3$ from exact Legendre moments.}
				\label{fig:LegendreNoise0BodyNr3}
\end{figure}
\begin{figure}[h!]
	\centering
		\includegraphics[width=1\textwidth]{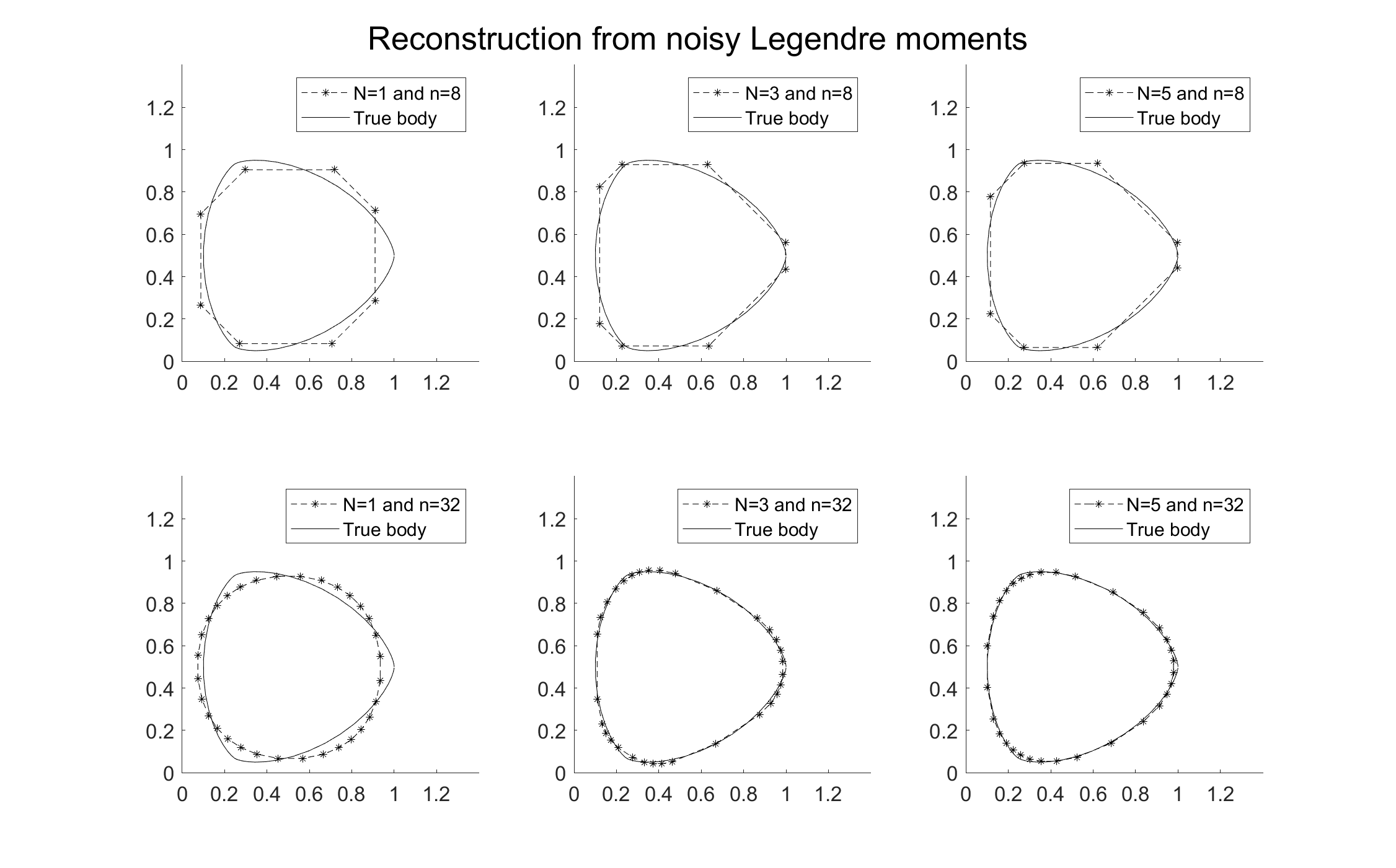}
		\caption{Reconstruction of the body $K_3$ from noisy Legendre moments.}
				\label{fig:LegendreNoise1BodyNr3}
\end{figure}

\section*{Acknowledgements}
The author Julia Schulte was supported by the German Research Foundation (DFG) via the Research Group FOR 1548 \grqq Geometry and Physics of Spatial Random Systems" and by ETH Foundations of Data Science.
The author Astrid Kousholt was supported by Centre for Stochastic Geometry and Advanced Bioimaging, funded by the Villum Foundation. We are very grateful to Markus Kiderlen for his ideas and useful comments during the process of writing this paper.  
%%%%%%%%%%%%%%%%%%%%%%%%%%%%%%%%%%%%%%%%%%%%%%%%%%%%%%%%%%%%%%%%%%%%

%\bibliographystyle{IEEEtranSA}
\bibliography{shapereconstruction2}

\section{Appendix}

\begin{lem}\label{Lem:GeometricMomentCoefficients}
For $1\leq i\leq n$ let $v_i$ be defined as in \eqref{vertexexpr} and let $0\leq k,l\leq N$ and $h\in\R^n$. Then we have 
\[
\sign(h_{i+1})\mu_{kl}(\conv\{0,v_i,v_{i+1}\})=\sum\limits_{q_1=0}^{k+l+1}\sum\limits_{q_2=0}^{k+l+2-q_1} M_{kl}(i,q_1,q_2)h_i^{q_1}h_{i+1}^{q_2} h_{i+2}^{k+l+2-q_1-q_2}
\]
where
\begin{align*}
M_{kl}(i,q_1,q_2)
&= \sum\limits_{q_1=0}^{k+l+1}\sum\limits_{q_2=1}^{k+l-q_1+2} h_i^{q_1} h_{i+1}^{q_2} h_{i+2}^{k+l+2-q_1-q_2}\\
&\quad \times \Big(a(i) \tilde{M}_{k,l}(i,q_1-1,q_2-1)\mathbf{1}\{q_1\neq 0\}\\
&\quad\quad+b(i)\tilde{M}_{k,l}(i,q_1,q_2-2)\mathbf{1}\{q_1\leq k+l,q_2\geq 2\}\\
&\quad\quad+c(i)\tilde{M}_{k,l}(i,q_1,q_2-1)\mathbf{1}\{q_1\leq k+l,1\leq q_2\leq k+l-q_1+1\}\Big)
\end{align*}
with 
\begin{align*}
&a(i)=-s_{i+1}c_{i+2}+c_{i+1}s_{i+2},\quad
b(i)=s_i c_{i+2}-c_i s_{i+2},\quad
c(i)=-s_i c_{i+1}+c_i s_{i+1}
\end{align*}
and
\begin{align*}
\tilde{M}_{k,l}(i,q_1,q_2)
&=\sum\limits_{p=0}^k\sum\limits_{q=0}^l\sum\limits_{r_1=0\vee(q_1-p)}^{q_1\wedge q}\sum\limits_{r_2=0\vee(p+l-q_1-q_2)}^{(k+l-q_1-q_2)\wedge(l-q)} \binom{k}{p}\binom{l}{q}(c_i s_{i+1}-s_i c_{i+1})^{-p-q-1}\\
&\quad\times (c_{i+1}s_{i+2}-s_{i+1}c_{i+2})^{-k+p-l+q-1}\frac{1}{(p+q+1)(k+l+2)}\binom{p}{q_1-r_1}\\
&\quad\times \binom{k-p}{k+l-q_1-q_2-r_2}\binom{q}{r_1}\binom{l-q}{r_2}(-1)^{p-q+q_2-k}\\
&\quad\times s_i^{p-q_1+r_1}s_{i+1}^{k+l-q_2-r_1-r_2}s_{i+2}^{q_1+q_2+r_2-l-p} c_i^{q-r_1}c_{i+1}^{r_1+r_2}c_{i+2}^{l-q-r_2}.
\end{align*}
\end{lem}
%---------------------------------------------------------------------------------------------------------------
\begin{proof}
By \eqref{GeometricMomentTriangle} we have  
\begin{align}
& \sign(h_{i+1}) \mu_{kl}(\conv\{0,v_i,v_{i+1}\})\notag
\\
 & =\int\limits_0^1\int\limits_0^{1-x_2} (v_{i,1}x_1+v_{i+1,1}x_2)^{k} (v_{i,2}x_1+v_{i+1,2}x_2)^l (v_{i,1}v_{i+1,2}-v_{i,2}v_{i+1,1})dx_1 dx_2\notag\\
 & = \int\limits_0^1\int\limits_0^{1-x_2}
 \sum\limits_{p=0}^k\sum\limits_{q=0}^l \binom{k}{p}\binom{l}{q}
 v_{i,1}^p v_{i+1,1}^{k-p} v_{i,2}^q v_{i+1,2}^{l-q} x_1^{p+q} x_2^{k+l-p-q}\notag\\
&\quad\times  (v_{i,1}v_{i+1,2}-v_{i,2}v_{i+1,1})dx_1 dx_2\notag\\
 & =  \sum\limits_{p=0}^k\sum\limits_{q=0}^l \binom{k}{p}\binom{l}{q}
 v_{i,1}^p v_{i+1,1}^{k-p} v_{i,2}^q v_{i+1,2}^{l-q} \sum\limits_{m=0}^{p+q+1}\binom{p+q+1}{m}\notag\\
&\quad\times \frac{(-1)^{p+q+1-m}}{(p+q+1)(k+l+2-m)} (v_{i,1}v_{i+1,2}-v_{i,2}v_{i+1,1})\label{GeoMomIdentity1}\\
  & =  \sum\limits_{p=0}^k\sum\limits_{q=0}^l \sum\limits_{q_1=0}^{p+q+1}\sum\limits_{q_2=p+q+1-q_1}^{k+l+2-q_1}
 \tilde{M}_{kl}(i,q_1,q_2) h_i^{q_1} h_{i+1}^{q_2} h_{i+2}^{k+l+2-q_1-q_2}\notag\\
  & =  \sum\limits_{q_1=0}^{k+l+1}\sum\limits_{q_2=0}^{k+l+2-q_1}
 M_{kl}(i,q_1,q_2) h_i^{q_1} h_{i+1}^{q_2} h_{i+2}^{k+l+2-q_1-q_2}.\notag
 \end{align}
To obtain the explicit expressions for the constants $\tilde{M}_{kl}(i,q_1,q_2)$ and $M_{kl}(i,q_1,q_2)$ we first observe that
\begin{align*}
v_{i,1}^p v_{i+1,1}^{k-p}&=\frac{1}{(c_i s_{i+1}-s_i c_{i+1})^p (c_{i+1}s_{i+2}-s_{i+1}c_{i+2})^{k-p}}(h_i s_{i+1}-h_{i+1}s_i)^p\\
&\quad\times (h_{i+1}s_{i+2}-h_{i+2}s_{i+1})^{k-p}\\
&=\frac{1}{(c_i s_{i+1}-s_i c_{i+1})^p (c_{i+1}s_{i+2}-s_{i+1}c_{i+2})^{k-p}}\\
&\quad\times \sum\limits_{m_1=0}^p\sum\limits_{m_2=0}^{k-p}\binom{p}{m_1}\binom{k-p}{m_2}
s_{i+1}^{m_1}(-s_i)^{p-m_1}s_{i+2}^{k-p-m_2}(-s_{i+1})^{m_2}\\
&\quad\times h_i^{m_1}h_{i+1}^{p-m_1}h_{i+2}^{m_2}h_{i+1}^{k-p-m_2}
\end{align*}
and
\begin{align*}
v_{i,2}^q v_{i+1,2}^{l-q}&=(c_i s_{i+1}-s_i c_{i+1})^{-q}(c_{i+1}s_{i+2}-s_{i+1}c_{i+2})^{-l+q}\\
&\quad\times (h_{i+1}c_i-h_i c_{i+1})^q(h_{i+2}c_{i+1}-h_{i+1}c_{i+2})^{l-q}\\
&=(c_i s_{i+1}-s_i c_{i+1})^{-q}(c_{i+1} s_{i+2}-s_{i+1}c_{i+2})^{-l+q}\\
&\quad\times\sum\limits_{r_1=0}^q\sum\limits_{r_2=0}^{l-q}\binom{q}{r_1}\binom{l-q}{r_2} c_i^{q-r_1}(-c_{i+1})^{r_1} c_{i+1}^{r_2}(-c_{i+2})^{l-q-r_2} h_i^{r_1} h_{i+1}^{q-r_1+l-q-r_2}h_{i+2}^{r_2}
\end{align*}
and
\begin{align*}
v_{i,1}v_{i+1,2}-v_{i,2}v_{i+1,1}&=\frac{1}{(c_i s_{i+1}-s_i c_{i+1})(c_{i+1}s_{i+2}-s_{i+1}c_{i+2})}\\
&\quad\times \Big[(h_i s_{i+1}-h_{i+1}s_i)(h_{i+2}c_{i+1}-h_{i+1}c_{i+2})
-(h_{i+1}c_i-h_ic_{i+1})\\
&\quad\quad\times(h_{i+1}s_{i+2}-h_{i+2}s_{i+1}))\Big]\\
&=\frac{1}{(c_i s_{i+1}-s_i c_{i+1})(c_{i+1}s_{i+2}-s_{i+1}c_{i+2})}\\
&\quad \times\Big[h_i h_{i+1}(-s_{i+1}c_{i+2}+c_{i+1}s_{i+2})+h_{i+1}^2(s_ic_{i+2}-c_i s_{i+2})\\
&\quad\quad+h_{i+1}h_{i+2}(-s_i c_{i+1}+c_i s_{i+1}\Big].
\end{align*}

Thus, we obtain
\begin{align}
\eqref{GeoMomIdentity1}&=\sum\limits_{p=0}^k\sum\limits_{q=0}^l \binom{k}{p}\binom{l}{q}(c_i s_{i+1}-s_i c_{i+1})^{-p-q-1}(c_{i+1}s_{i+2}-s_{i+1}c_{i+2})^{-k+p-l+q-1}\notag\\
&\quad\times \Big[h_i h_{i+1}(-s_{i+1}c_{i+2}+c_{i+1}s_{i+2})+h_{i+1}^2(s_i c_{i+2}-c_i s_{i+2})\notag\\
&\quad\quad+h_{i+1}h_{i+2}(-s_i c_{i+1}+c_i s_{i+1})
\Big]\sum\limits_{m=0}^{p+q+1}\binom{p+q+1}{m}\frac{(-1)^{p+q+1-m}}{(p+q+1)(k+l+2-m)}\notag\\
&\quad\quad+ \sum\limits_{m_1=0}^p\sum\limits_{m_2=0}^{k-p}\sum\limits_{r_1=0}^q
\sum\limits_{r_2=0}^{l-q}\binom{p}{m_1}\binom{k-p}{m_2} s_{i+1}^{m_1} (-s_i)^{p-m_1}s_{i+2}^{k-p-m_2} (-s_{i+1})^{m_2}\notag\\
&\quad\times\binom{q}{r_1}\binom{l-q}{r_2} c_i^{q-r_1}
(-c_{i+1})^{r_1} c_{i+1}^{r_2} (-c_{i+2})^{l-q-r_2} h_i^{m_1+r_1} h_{i+1}^{k-m_1-m_2+l-r_1-r_2} h_{i+2}^{m_2+r_2}\notag\\
&= \sum\limits_{p=0}^k\sum\limits_{q=0}^l\sum\limits_{m_1=0}^p\sum\limits_{m_2=0}^{k-p}\sum\limits_{r_1=0}^q\sum\limits_{r_2=0}^{l-q}\binom{k}{p}\binom{l}{q}
\binom{p}{m_1}\binom{k-p}{m_2}\binom{q}{r_1}\binom{l-q}{r_2}\notag\\
&\quad\times (c_i s_{i+1}-s_i c_{i+1})^{-p-q-1} (c_{i+1}s_{i+2}-s_{i+1}c_{i+2})^{-k+p-l+q-1} \frac{1}{(p+q+1)(k+l+2)}\notag\\
&\quad\times (-1)^{p-m_1-m_2+l-q-r_1-r_2} s_i^{p-m_1}s_{i+1}^{m_1+m_2}s_{i+2}^{k-p-m_2}c_i^{q-r_1}c_{i+1}^{r_1+r_2}c_{i+2}^{l-q-r_2}h_i^{m_1+r_1}\notag\\
&\quad\times h_{i+1}^{k-m_1-m_2+l-r_1-r_2} h_{i+2}^{m_2+r_2}\Big[h_i h_{i+1} (-s_{i+1}c_{i+2}+c_{i+1}s_{i+2})+h_{i+1}^2(s_i c_{i+2}-c_i s_{i+2})\notag\\
&\quad\quad+h_{i+1}h_{i+2}(-s_i c_{i+1}+c_i s_{i+1})\Big]\label{GeoMomIdentity2}.
\end{align}
Now, we introduce new indices $q_1=m_1+r_1$ and $q_2=k+l-m_1-m_2-r_1-r_2$ with summation $0\leq q_1\leq k+l$ and $0\leq q_2\leq k+l-q_1$. 
The new summation range of the indices $r_1$ and $r_2$ are then $0\vee(q_1-p)\leq r_1\leq q_1\wedge q$ and $0\vee (p+l-q_1-q_2)\leq r_2\leq (k+l-q_1-q_2)\wedge (l-q)$. 
The index change yields
\begin{align*}
\eqref{GeoMomIdentity2}&=\sum\limits_{q_1}^{k+l}\sum\limits_{q_2=0}^{k+l-q_1}h_i^{q_1}h_{i+1}^{q_2}h_{i+2}^{k+l-q_1-q_2}\tilde{M}_{k,l}(i,q_1,q_2)
\Big[h_i h_{i+1}a(i)+h_{i+1}^2 b(i)+ h_{i+1}h_{i+2} c(i)\Big]\\
&= \sum\limits_{q_1=0}^{k+l}\sum\limits_{q_2=0}^{k+l-q_1}\Big(h_i^{q_1+1}h_{i+1}^{q_2+1}h_{i+2}^{k+l-q_1-q_2} a(i)+h_i^{q_1}h_{i+1}^{q_2+2}h_{i+2}^{k+l-q_1-q_2}b(i)\\
&\quad+h_i^{q_1}h_{i+1}^{q_2+1}h_{i+2}^{k+l-q_1-q_2+1}c(i)\Big)\tilde{M}_{k,l}(i,q_1,q_2)
\end{align*}
which implies the assertion.
\end{proof}
\end{document}